\documentclass[a4paper,reqno,12pt,oneside]{amsart}
\usepackage{amsmath,geometry}
\usepackage{graphicx}
\usepackage{mathrsfs}
\usepackage{amssymb,amsmath, amsfonts, amsthm,color}
\usepackage{latexsym}
\usepackage{color} 
\usepackage[dvips]{epsfig}
\usepackage[colorlinks]{hyperref}
\usepackage{comment}
\hypersetup{
    colorlinks,
    citecolor=blue,
    linkcolor=blue,
    urlcolor=black
}
\usepackage{tikz}
\usepackage{pgfplots}
\usetikzlibrary{patterns}
\usepackage{bm}
\newtheorem{theorem}{Theorem}[section]
\newtheorem{lemma}[theorem]{Lemma}

\newtheorem{proposition}[theorem]{Proposition}

\newtheorem{remark}[theorem]{Remark}

\numberwithin{equation}{section}

\newcommand{\e}{\varepsilon}

\newcommand{\R}{\mathbb{R}}

\newcommand{\de}{\partial}

\newcommand{\D}{{\mathfrak{D}}}

\newcommand{\dsh}{{2^\sharp}}
\newcommand{\dst}{{2^*}}

\newcommand*{\abs}[1]{\left\vert #1\right\vert}

\newcommand{\beq }{\begin{equation}}
	\newcommand{\eeq }{\end{equation}}

\title[Blow-up phenomena for a boundary Yamabe problem ]{Blow-up phenomena for a boundary Yamabe problem with umbilic boundary}

\author{Giusi Vaira}

\thanks{Work partially supported by 
the MUR-PRIN-P2022YFAJH ``Linear and Nonlinear PDE's: New directions and Applications" and by the INdAM-GNAMPA project ``Fenomeni non lineari: problemi locali e non locali e loro applicazioni",  CUP E5324001950001}

\begin{document}
\date{}
\subjclass[2010]{35B44, 58J32}
\keywords{Prescribed curvature problem, conformal metric,  clustering blow-up point}
 \thanks{}

\maketitle
\begin{abstract}
We consider a linear perturbation of the classical geometric problem of prescribing the scalar and the boundary mean curvature problem in a Riemannian manifold with umbilic boundary provided the Weyl tensor is non-zero everywhere. We will deal with the case of negative scalar curvature showing the existence of a positive solutions when $n\geq 8$. \end{abstract}

\section{Introduction}
One of the most important problems in differential geometry is the so-called prescribed curvature problem, i.e.  {\em given $(M,g)$ be a Riemannian closed manifold of dimension $n\geq3$ and a smooth function $\mathbf{K}:M\to\mathbb R$, finding a metric $\tilde g$ conformal to the original metric $g$ whose scalar curvature is $\mathbf{K}$} (see \cite{SchoenYau_1994,ChenLi_2010,Hebey_2000,KazdanWarner_1974}).\\  As it is well known, being $\tilde g=u^{\frac4{n-2}}g$, this is equivalent to finding a positive solution of the semi-linear elliptic equation:
\begin{align}\label{eq.curvature_problem_no_boundary}
-\frac{4(n-1)}{n-2}\Delta_gu+k_gu=\mathbf{K}u^{\frac{n+2}{n-2}},\qquad~ u>0,\qquad&\text{in}~M,
\end{align}
where $k_g$ denotes the scalar curvature of $M$ with respect to $g$ and $\Delta_g$ is the Beltrami-Laplace operator.

If $M$ is a manifold with boundary, given a smooth function $\mathbf{H}:\partial M\to\mathbb R$, it is natural to ask if there exists a conformal metric whose scalar curvature and boundary mean curvature can be prescribed as $\mathbf{K}$ and $\mathbf{H}$ respectively. As in \eqref{eq.curvature_problem_no_boundary}, the geometric problem turns out to be equivalent to a semi-linear elliptic equation with a Neumann boundary condition:
\begin{align}\label{eq.basicly_original}
\left\{\begin{array}{ll}
-\frac{4(n-1)}{n-2}\Delta_gu+k_gu=\mathbf{K}u^{\frac{n+2}{n-2}},\,\quad u>0,&\text{in }M,\\\\\frac2{n-2}\partial_\nu u+h_gu=\mathbf{H}u^{\frac n{n-2}},&\text{on }\partial M,
\end{array}\right.
\end{align}
where, $k_g$ and $h_g$ denote the scalar and boundary mean curvatures of $M$ with respect to $g$ and $\nu$ is the outward normal unit vector with respect to the metric $g$.

When $\mathbf{K}$ and $\mathbf{H}$ are constants, the problem is known as the Escobar problem, since it was first proposed and studied by Escobar in 1992 in the case $\mathbf{H}=0$ (\cite{Escobar_1992,Escobar_1996}) and in the case $\mathbf{K}=0$ (\cite{Escobar_1992a}). 
Afterwards, many subsequent contributions for this problem are given in \cite{Almaraz, BrendleChen, MayerNdiaye_2017,Marques}.

The case of non-zero constants $\mathbf{K}$ and $\mathbf{H}$ (with $\mathbf{K}>0$) it was first studied by Han \& Li in \cite{HanLi_1999,HanLi_2000} and then it was completed by Chen, Ruan \& Sun in \cite{ChenRuanSun}.\\ In all these results, the existence of solutions for the problem \eqref{eq.basicly_original} strongly depends on the dimension of the manifold, on the properties of the boundary (i.e. being umbilic or not) and on vanishing properties of the Weyl tensor.\\

The case of non-constant functions $\mathbf{K}$ and $\mathbf{H}$ is less studied and all the available results are for special manifolds (tipically the unit ball and the half-sphere). We refer to \cite{Li_1995,Li_1996,AhmedouBenAyed_2021,BenAyedElMehdiOuldAhmedou_2005,BenAyedElMehdiAhmedou_2002} for the case $\mathbf{H}=0$ and to \cite{AbdelhediChtiouiAhmedou_2008,DjadliMalchiodiAhmedou_2004,XuZhang_2016,ChangXuYang_1998} for  the case $\mathbf{K}=0$.\\

When both $\mathbf{K}$ and $\mathbf{H}$ are not constants and not zero, the problem becomes more difficult. Djadli, Malchiodi \& Ahmedou consider problem \eqref{eq.basicly_original} in \cite{DjadliMalchiodiOuldAhmedou_2003} on the three-dimensional half-sphere proving some existence and compactness results. Chen, Ho \& Sun  proved the existence of solutions for \eqref{eq.basicly_original} when $\mathbf{K}$ and $\mathbf{H}$ are negative functions and the boundary $\partial M$ has negative Yamabe invariant (see \cite{ChenHoSun_2018}).
In  \cite{AmbrosettiLiMalchiodi_2002}, Ambrosetti, Li \& Malchiodi considered the perturbation problem in the unit ball when both $\mathbf{K}$ and $\mathbf{H}$ are positive. That is, they consider $\mathbf{K} = \mathbf{K}_0 + \varepsilon \mathcal{K} > 0$ and $\mathbf{H} = \mathbf{H}_0 + \varepsilon \mathcal{H} > 0$, where $\mathbf{K}_0 > 0$, $\mathbf{H}_0 > 0$, $\varepsilon > 0$ is small, and $\mathcal{K}$ and $\mathcal{H}$ are smooth functions. They proved an existence result when $\mathcal{K}$ and $\mathcal{H}$ satisfy some conditions.\\
The first result concerning the case of negative prescribed scalar curvature (namely $\mathbf{K}<0$) is due to Cruz-Bl\'azquez, Malchiodi \& Ruiz in \cite{CruzMalchiodiRuiz}. They introduce the scaling invariant
quantity $$\mathfrak D:=\sqrt{n(n-1)}\frac{H(p)}{\sqrt{|\mathbf K(p)|}},\quad p\in\partial M$$ and established the existence of a solution to \eqref{eq.basicly_original} whenever $\mathfrak D<1$ along the whole boundary. On the other hand, if $\mathfrak D>1$ at some boundary points, they got a solution only in a three dimensional manifold, for a generic choice of $\mathbf K$ and $\mathbf H$.\\ Their proof relies on a careful blow-up analysis: first they show that the blow-up phenomena occurs
at boundary points $p$ with $\mathfrak D\geq1$, with different behaviours depending on whether $\mathfrak D=1$ or $\mathfrak D>1$. To deal with the loss of compactness at points with $\mathfrak D>1$, where bubbling of solutions occurs, it is shown that in dimension three all the blow-up points are isolated and simple. As a consequence, the number of blow-up points is finite and the blow-up is excluded via integral estimates. In that regard, $n= 3$ is the maximal dimension for which one
can prove that the blow-up points with $\mathfrak D>1$ are isolated and simple for generic choices of $\mathbf K$ and $\mathbf H$. In the closed case such a property is assured up to dimension four (see \cite{Li_1996}) but, as observed in \cite{DjadliMalchiodiOuldAhmedou_2003}, the
presence of the boundary produces a stronger interaction of the bubbling solutions with the function $\mathbf K$.\\

Afterwards, in \cite{BattagliaCruz-BlazquezPistoia_2023}, the authors considered the perturbation problem in the ball under the condition $\mathbf{K}<0$ and $\mathbf{H} > 0$. i.e., $\mathbf{K} = \mathbf{K}_0 + \varepsilon \mathcal{K} < 0$ and $\mathbf{H} = \mathbf{H}_0 + \varepsilon \mathcal{H} > 0$, where $\mathbf{K}_0 < 0$, $\mathbf{H}_0 > 0$, $\varepsilon > 0$ is small, and $\mathcal{K}$ and $\mathcal{H}$ are smooth functions showing the existence of solutions with some constraint of $\mathcal{K}$ and $\mathcal{H}$.\\
Recently, in \cite{BPV} it is consider problem 
 \eqref{eq.basicly_original} in the unit ball showing the existence of infinitely many non-radial solutions under some suitable assumptions on the functions $\mathbf K$ and $\mathbf H$ (see also \cite{WeiYan_2010a}) for the closed case, \cite{WangZhao_2013} and \cite{BianChenYang} for $\mathbf K=0$ and $\mathbf H>0$). \\ The existence in the general case in dimension $n\geq 4$ is not known at the moment since the difficulties that arise in order to prove the compactness condition. 
In \cite{Cruz-BlazquezPistoiaVaira_2022} and in 
\cite{Cruz-BlazquezVaira_2025} it is studied a linear perturbation of the geometric problem \eqref{eq.basicly_original}, namely 
\begin{equation}\label{pb}
\left\{\begin{aligned}&-\frac{4(n-1)}{n-2}\Delta_g v +k_g v =\mathbf K v^{\frac{n+2}{n-2}}\quad&\mbox{in}\,\, M\\
&\frac{2}{n-2}\frac{\partial v}{\partial \nu}+h_g v +\varepsilon \gamma v= \mathbf H v^{\frac{n}{n-2}}\quad &\mbox{on}\,\, \partial M.\end{aligned}\right.\end{equation}
where $\varepsilon$ is a small and positive parameter and $\gamma$ is a given smooth function. By using the Escobar metric (namely by letting $g$ so that $h_g=0$) then, in \cite{Cruz-BlazquezPistoiaVaira_2022}, it is shown the existence of a clustering type solutions for problem \eqref{pb} in the case in which $\mathbf K$ and $\mathbf H$ are constants, $4\leq n\leq 7$ and $\gamma=1$, while, in \cite{Cruz-BlazquezVaira_2025}, it is proved the existence of a blowing-up solution when $\mathbf K$ and $\mathbf H$ are not constants, $n\geq 4$ and $\gamma=1$.\\ Here we continue the study of the problem \eqref{pb} when the manifold $M$ has an umbilic boundary and we want to show the existence of a blowing up solution also in this case. Here we let
\begin{itemize}
\item[${\rm (Hyp_1)}$] 
$\mathbf {K}, \mathbf{H}$ are sufficiently regular functions such that
$\mathbf K<0$,
$\mathbf H$ is of arbitrary sign and there exists
$p\in \de M$ with
$\D>1$. 
\item [${\rm (Hyp_2)}$] 
$p$ is a common local minimum point 
which is non-degenerate, i.e.,
$\nabla \mathbf K(p)=\nabla \mathbf H(p)=0$ and
$D^2 \mathbf H(p)$ and
$D^2\mathbf K(p)$ are positive definite.
\end{itemize} We remark that, if $\mathbf K$ and $\mathbf H$ are constants, ${\rm (Hyp_1)}$ means only that $\mathbf K<0$ and $\mathbf H>0$ are such that $\D>1$.\\ 
The main result of the paper is stated as follows.
\begin{theorem}\label{principale}
Let $(M, g)$ be a smooth, $n-$ dimensional Riemannian manifold of positive type with regular umbilic boundary $\partial M$. Suppose $n\geq 8$ and that the Weyl tensor is not zero everywhere on $\partial M$.  Assume ${\rm (Hyp_1)}$.\\ If $\mathbf K$ and $\mathbf H$ are constants we let
 $\gamma: M\to \mathbb R$ a smooth function, $\gamma>0$ on $\partial M$,  while if $\mathbf K$ and $\mathbf H$ are not constants we let $\gamma=1$ and we assume ${\rm (Hyp_2)}$.\\ Then, for $\varepsilon>0$ small there exists a positive solution $u_\e$ that blows up at a point $p\in\partial M$ as $\e\to 0$.   
 \end{theorem}
Let us make some comments.
\begin{itemize} 
\item The proof of Theorem \ref{principale} relies on a finite dimensional Lyapunov-Schmidt reduction method. Here the main difficulty is due to the fact that the umbilicity of the boundary forces to consider higher order expansion in the metric $g$ that, together with a different kind of bubble, makes the computations so hard. Moreover, when $\mathbf K$ and $\mathbf H$ are constants we also need to correct the main part of the ansatz by adding a function $V_p$ (given in Proposition \ref{vp}) in order to have a good error. We remark that, when $\mathbf K$ and $\mathbf H$ are not constants, then it is not needed the correction and this is a great difference with respect to the non-umbilic case (the result contained in \cite{Cruz-BlazquezVaira_2025} with $\mathbf K$ and $\mathbf H$ not constants need the correction which is different from $V_p$) .
\item The result provide the exact location of the blow-up point when $\mathbf K$ and $\mathbf H$ are not constants. Indeed, in this case, the point is the common non-degenerate critical point. Here $\gamma$ can be taken equal to $1$ since it does'n have any role. When $\mathbf K$ and $\mathbf H$ are constants the situation is a little bit complicated and it is not possible to give the precise location of the point in which the blow-up occurs although we strongly believe that the geometric function which is responsable of the existence of the blowing-up solution is the Weyl tensor on the boundary of the manifold. However, to capture the geometry of the manifold, we need to have an explicit form of the function $V_p$ given in Proposition \ref{vp} which is far from being possible.
\item In \cite{GMP} it was considered the problem with $\mathbf K=0$. We remark that even if one can think that these problems are similar, the form of the bubble, namely the classification result for the limit problem, makes the computations  completely different. 
\end{itemize}	
The structure of the paper is the following: first, in Section \ref{preliminari} we collect some useful notations and results, then, in Section \ref{ansatz} we find a good approximated solution. Next we reduce the problem into a finite dimensional one (see Section \ref{riduzione}) and finally, in Section \ref{ridotto}, we study the reduced problem and we prove the Theorem \ref{principale}.  

\section{Preliminaries and Variational Framework}\label{preliminari}
{\it Notations:} Here we collect our main notations. \\ We will use the indices $1\leq i, j, k, m, p, r, s\leq n-1$ and $1\leq a, b, c, d\leq n$. \\ We denote by $g$ the Riemannian metric, by $R_{abcd}$ the full Riemannian curvature tensor, by $R_{ab}$ the Ricci tensor and by $k_g$ the scalar curvature of $(M, g)$. Moreover the Weyl tensor of $(M, g)$ will be denoted by ${\rm Weyl}_g$.\\ Let $(h_{ij})_{ij}(p)$ be the tensor of the second fundamental form in a point $p\in\partial M$. We recall that the boundary $\partial M$ is umbilic (namely composed only of umbilic points) when, for all $p\in\partial M$, $h_{ij}(p)=0$ for all $i\neq j$ and $h_{ii}(p)=h_g(p)$ where $h_g(p)$ is the mean curvature of $\partial M$ at the point $p$.\\
The bar over an object (e.g. $\overline{{\rm Weyl}_g}$) will mean the restriction to this object to the metric of $\partial M$. We will often use the notation $$\mathcal L_g:=-\frac{4(n-1)}{(n-2)}\Delta_g+k_g,\quad \mathcal B_g:=\frac{\partial}{\partial\nu}+\frac{n-2}{2}h_g$$ to denote the conformal laplacian and the conformal boundary operator respectively.\\
When we derive a tensor, e.g. $T_{ij}$, with respect to a coordinate $y_\ell$ we use the usual shortened notation $T_{ij, \ell}$ for $\frac{\partial }{\partial y_\ell}T_{ij}$.\\ Finally, for a tensor $T$ and a number $q\in\mathbb N$, we use $${\rm Sym}_{i_1\ldots i_g}T_{i_1\ldots i_q}=\frac{1}{q!}\sum_{\sigma\in \mathtt S_q}T_{i_{\sigma(1)}\ldots i_{\sigma(q)}}$$ being $\mathtt S_q$ the group of all permutations of $q$ elements.
\begin{remark}\label{umbilic}
Since $\partial M$ is umbilic for any $p\in\partial M$ there exists a metric $\tilde g_p=\tilde g$ conformal to $g$, namely $\tilde g_p=\Lambda^{\frac{4}{n-2}}g$ such that 
\begin{equation}\label{detg}
|{\rm det }\tilde g_p(y)|=1+\mathcal O(|y|^n)
\end{equation}\\
\begin{equation}\label{hij}
|\tilde h_{ij}(y)|=o(|y|^3)
\end{equation}
\\
\begin{equation}\label{gij}
\begin{aligned}
\tilde g^{ij}(y)&=\delta_{ij}+\frac 13 \bar R_{ikj\ell}y_k y_\ell + R_{ninj}y_n^2\\
&+\frac 16 \bar R_{ikj\ell, m}y_ky_\ell y_m +R_{ninj, k}y_n^2 y_k+\frac 13 R_{ninj, n}y_n^3\\
&+\left(\frac{1}{20} \bar R_{ikj\ell, mp}+\frac{1}{15}\bar R_{iks\ell}\bar R_{jmsp}\right)y_ky_\ell y_m y_p\\
&+\left(\frac 12 R_{ninj, k\ell}+\frac 13{\rm Sym}_{ij}(\bar R_{iks\ell}R_{nsnj})\right)y_n^2y_ky_\ell\\
&+\frac 13 R_{ninj, nk}y_n^3y_k+\frac{1}{12}(R_{ninj, nn}+8R_{nins}R_{nsnj})y_n^4+\mathcal O(|y|^5)\end{aligned}\end{equation}\\
\begin{equation}\label{barSg}
|\bar{k}_{\tilde g_p}(y)|=\mathcal O(|y|^2)\quad\hbox{and}\quad \partial_{ii}^2 \bar{k}_{\tilde g_p}(p)=-\frac 16 |\overline{{\rm Weyl}_g}(p)|^2
\end{equation}\\
\begin{equation}\label{Sg}
{k}_{\tilde g_p}(p)=0; \, k_{\tilde g_p,a}(p)=0, \quad\hbox{and}\quad \partial_{ii}^2 {k}_{\tilde g_p}(p)=-\frac 16 |\overline{{\rm Weyl}_g}(p)|^2
\end{equation}\\

\begin{equation}\label{Rij}
\bar R_{k\ell}(p)=R_{nn}(p)=R_{nk}(p)=0
\end{equation}
uniformly with respect to $p\in\partial M$ and $y\in T_pM$. Also $\Lambda_p(p)=1$ and $\nabla \Lambda_p(p)=0$. These results are contained in \cite{Marques}.
\end{remark} 
The conformal laplacian and the conformal boundary operator transform under the change of metric $\tilde g_p=\Lambda_p^{\frac{4}{n-2}}g$ in the following way:
$$\begin{aligned}&\mathcal L_{\tilde g_p}\varphi=\Lambda_p^{-\frac{n+2}{n-2}}\mathcal L_g(\Lambda_p \varphi)\\ &\mathcal B_{\tilde g_p}\varphi=\Lambda_p^{-\frac{n}{n-2}}\mathcal B_g(\Lambda_p \varphi).\end{aligned}$$ Then we can rewrite our inital problem \eqref{pb} in the following way: let $v:=\Lambda_p u$. Then $v$ is a positive solution of \eqref{pb} if and only if $u$ is a positive solution of 
\begin{equation}\label{pb1}
\left\{\begin{aligned}&\mathcal L_{\tilde g_p}u =\mathbf K u^{\frac{n+2}{n-2}}\quad&\mbox{in}\,\, M\\
&\mathcal B_g u+\varepsilon \Lambda_p^{-\frac{2}{n-2}}\gamma u= \mathbf H u^{\frac{n}{n-2}}\quad &\mbox{on}\,\, \partial M.\end{aligned}\right.\end{equation}
From now on we set $\tilde\gamma:=\Lambda_p^{-\frac{2}{n-2}}\gamma $.\\
We endow the Sobolev space
$H^1_g(M)=H^1(M)$ the equivalent scalar product
$$\langle u, v\rangle_{g} :=\int_M 
( c_n \nabla_g u\nabla_g v + k_g uv ) \, d\nu_g+2(n-1)\int_{\partial M}h_g uv\, d\sigma_g,$$ where
$d\nu_g$ is the volume element of the 
manifold, $d\sigma_g$ is the volume element of the boundary
and
$c_n:=\frac{4(n-1)}{(n-2)}$. This scalar 
product induces a norm in
$H^1(M)$ which is equivalent to the 
standard one, and that we denote by
$\|\cdot\|_{g}$. We remark also that $\Lambda_p$ is an isometry in the sense that for any $u, v\in H^1(M)$ $$\langle \Lambda_p u, \Lambda_p v\rangle_g =\langle u, v\rangle_{\tilde g_p}$$ and, consequently $$\|\Lambda_p u\|_g=\|u\|_{\tilde g_p}.$$
Moreover, for any
$u\in L^q(M)$ and
$v\in L^q(\partial M)$, we put
$$
\|u\|_{L^q(M)}:=
\Big ( \int_M |u|^q\, d\nu_g
\Big ) ^{\frac 1 q} 
\quad\text{and}
\quad \|v\|_{L^q(\partial M)}:=
\Big ( \int_{\partial M}|v|^q\, d\sigma_g
\Big ) ^{\frac 1q}.
$$ 
 For notational convenience, 
we will often omit the volume or surface 
elements in integrals. 
 \\

We have the well-known embedding continuous maps
\begin{align*}
&\mathfrak i_{\partial M}: 
H^1(M)\to L^{\dsh}(\partial M),\qquad\qquad 
& \mathfrak i_M: H^1(M) \to L^{\dst}(M),\\ 
&\mathfrak i^*_{\partial M}: L^{\frac{2(n-1)}{n}}
(\partial M)\to H^1(M),
\qquad\qquad & \mathfrak i^*_M: 
L^{\frac{2n}{n+2}}(M)\to H^1(M),
\end{align*}
where
$\dst=\frac{2n}{n-2}$ and
$\dsh = \frac{2(n-1)}{n-2}$ 
denote the critical Sobolev exponents for
$M$ and
$\de M$, respectively.
 Now, given
$\mathfrak f\in L^{\frac{2(n-1)}{n}}(\partial M)$,
 the function
$w_1=\mathfrak i^*_{\partial M}(\mathfrak f)$ 
in
$H^1_g(M)$ is defined as the unique solution of 
the equation 
\begin{equation*}
 \begin{cases}
 \displaystyle
 \mathcal L_g w_1=0
  &\mbox{in}\,\, M,\\[6pt]
   \displaystyle
 \mathcal B_g w_1
=\mathfrak f
  &\mbox{on}\,\, \partial M.
\end{cases}
\end{equation*}
Similarly, if we let
$\mathfrak g\in L^{\frac{2n}{n+2}}(M)$,
$w_2=\mathfrak i^*_M(\mathfrak g)$ 
denotes the unique solution of the equation
\begin{equation*}
 \begin{cases}
  \displaystyle
\mathcal L_g w_2
=\mathfrak g
  &\mbox{in}\,\, M,\\[6pt]
   \displaystyle
\mathcal B_g w_2=0
  &\mbox{on}\,\, \partial M.
\end{cases}
\end{equation*}
By continuity of
$\mathfrak i_M$ and
$\mathfrak i_{\partial M}$,
 we get
$$
\|\mathfrak i^*_{\partial M}
(\mathfrak f)\|_{g}
\leq C_1 \|\mathfrak f\|_{L^{\frac{2(n-1)}{n}}
(\partial M)}
\quad\text{and}\quad 
\|\mathfrak i^*_M(\mathfrak g)\|_{g}
\leq C_2 \|\mathfrak g\|_{L^{\frac{2n}{n+2}}(M)},
$$ 
for some
$C_1>0$ and independent of
$\mathfrak f$ and some
$C_2>0$ and independent of
$\mathfrak g$. 
Then, we are able to rewrite the 
problem \eqref{pb} as 
\begin{equation}
\label{pb1} 
v=\mathfrak i^*_M(\mathbf K\mathfrak g(v))
+\mathfrak i^*_{\partial M}
\Big ( \frac{n-2}{2}
\Big ( \mathbf H \mathfrak f(v)-\varepsilon \gamma v
\Big )  
\Big ) ,
\end{equation} 
where we set
$\mathfrak g(v):=(v^+)^{\frac{n+2}{n-2}}$ and
$\mathfrak f(v)=(v^+)^{\frac{n}{n-2}}$.

We also define the energy
$J_{\varepsilon, g}: H^1(M)\to\mathbb R$ associated to 
\begin{equation}
\label{energia}
\begin{split}
J_{\varepsilon, g}(v)
 : & =\int_M 
\Big ( \frac{c_n}{2}|\nabla_g v|^2
+\frac 12k_g v^2 -\mathbf K \mathfrak G(v)
\Big ) \,d\nu_g +(n-1) \int_{\partial M} h_g v^2\, d\sigma_g\\
& -c_n\frac{n-2}{2}
\int_{\partial M}\mathbf H 
\mathfrak F(v)\, d\sigma_g
 +(n-1)
\varepsilon\int_{\partial M}\gamma
v^2\,d\sigma_g,
 \end{split}
 \end{equation}
being
$$
\mathfrak G(s)=\int_0^s \mathfrak g(t)\, dt,
\qquad 
\mathfrak F(s)=\int_0^s \mathfrak f(t)\,dt.
$$ 
Notice that, if we define
\begin{equation}
\label{energia1}
\begin{split}
\tilde J_{\varepsilon, \tilde g_p}(u)
 : & =\int_M 
\Big ( \frac{c_n}{2}|\nabla_{\tilde g_p} u|^2
+\frac 12 k_{\tilde g_p} u^2 -\mathbf K \mathfrak G(u)
\Big ) \,d\nu_{\tilde g_p} +(n-1) \int_{\partial M} h_{\tilde g_p} u^2\, d\sigma_{\tilde g_p}\\
& -c_n\frac{n-2}{2}
\int_{\partial M} \mathbf H 
\mathfrak F(u)\, d\sigma_{\tilde g_p}
 +(n-1)
\varepsilon\int_{\partial M}\tilde\gamma
u^2\,d\sigma_{\tilde g_p},
 \end{split}
 \end{equation}
 then we have \begin{equation}\label{relenergia}
 J_{\varepsilon, g}(\Lambda_p u)=\tilde J_{\varepsilon, \tilde g_p}(u).\end{equation}

 Now we introduce some integral quantities that will appear in our computations: let
\begin{equation*}I_m^\alpha:=\int_0^{+\infty}\frac{\rho^\alpha}{(1+\rho^2)^m}\, d\rho,\quad \hbox{with}\, \alpha+1<2m\end{equation*} It is useful to recall the following relations:
\begin{equation}\label{prop}
\begin{split} 
I_n^n=\frac{n-3}{n+1}I_n^{n+2}, \quad I_{n-2}^{n-2}=\frac{4(n-2)}{n+1}I_n^{n+2}.\end{split} 
\end{equation}
Moreover, for $p\in\partial M$ with $\D(p)>1$, we denote by
\begin{equation*}
	\varphi_m(p):=\int_{\D}^{+\infty}\frac{1}{(t^2-1)^m}\, dt;\quad {\hat\varphi}_m(p):=\int_{\D}^{+\infty}\frac{(t-\D)^2}{(t^2-1)^m}\, dt; \quad {\tilde\varphi}_m(p):=\int_{\D}^{+\infty}\frac{(t-\D)^4}{(t^2-1)^m}\, dt .
\end{equation*}

\medskip

Here and in the sequel we agree that $f\lesssim g$  means $|f|\leq C |g| $ for some positive constant $c$ which is independent on $f$ and $g$ and $f\sim g$  means $f= g(1+o(1))$. We use the letter $C$ to denote a positive constant that may change from line to line.

\section{The ansatz}\label{ansatz}

We want to find a solution $u$ of the problem \eqref{pb1} by a finite dimensional reduction.\\
The main ingredient to 
cook up our solution is the so-called 
\textit{bubble}, 
whose expression is given by
\begin{equation*}
U_{\delta, x_0(\delta)}(x):=
\frac{\alpha_n}{|\mathbf K(p)|^{\frac{n-2}{4}}}
\frac{\delta^{\frac{n-2}{2}}}{
\big ( |x-x_0(\delta)|^2-\delta^2
\big ) ^{\frac{n-2}{2}}}
\end{equation*} 
where
$\alpha_n:=
\big ( 4n(n-1)
\big ) ^{\frac{n-2}{4}}$,
$x_0(\delta):=
(\tilde x_0,-\D\delta) 
 \in \mathbb R^n$,
$\tilde x_0\in \R^{n-1}$ 
and
$\delta>0$. When
$\D>1$, the
$n-$dimensional family of 
functions defined above describe all the 
solutions to the following problem in
$\R^n_+$ (see
\cite{ChipotShafrirFila_1996}):
\begin{equation}\label{pblim}
\begin{cases}
\displaystyle
-c_n \Delta U =-|\mathbf K(p)|U^{\frac{n+2}{n-2}}
\quad 
&\mbox{in}\,\, \mathbb R^n_+
\\[6pt]
\displaystyle
\frac{2}{n-2}\frac{\partial U}{\partial\nu}
=\mathbf  H(p) U^{\frac{n}{n-2}}\quad &\mbox{on}\,\,
\partial\mathbb R^n_+.
\end{cases}
\end{equation}
We set \begin{equation}\label{U}U(x)=
U_{1, x_0(1)}(\tilde x, x_n)=\frac{\alpha_n}
{|\mathbf  K(p)|^{\frac{n-2}{4}}}\frac{1}{
\big ( |\tilde x|^2+(x_n+\mathfrak D)^2-1
\big ) ^{\frac{n-2}{2}}},\end{equation} where
$\tilde x=(x_1, \ldots, x_{n-1})\in \mathbb R^{n-1}$
 and
$x_n>0$. Moreover $\alpha_n:=\left(4n(n-1)\right)^{\frac{n-2}{4}}$.\\

We also need to introduce the linear problem 
\begin{equation}\label{lineare}
\begin{cases}
\displaystyle
-c_n\Delta v +|\mathbf  K(p)|\frac{n+2}{n-2}U^{\frac{4}{n-2}}v=0
\quad &\mbox{in}\,\, \mathbb R^n_+\\[6pt]
\displaystyle 
\frac{2}{n-2} \frac{\partial v}{\partial\nu}-\frac{n}{n-2}\mathbf  H(p)U^{\frac{n}{n-2}}v=0
\quad &\mbox{on}\,\, \partial\mathbb R^n_+.
\end{cases}
\end{equation}
In Section 2 of
\cite{Cruz-BlazquezPistoiaVaira_2022} we have shown that the
$n-$ dimensional space of solutions of \eqref{lineare} 
 is generated by the functions 
\begin{equation}\label{Ji}
\mathfrak j_i(x):=\frac{\partial U}{\partial x_i}(x)=
\frac{\alpha_n}{|\mathbf  K(p)|^{\frac{n-2}{4}}}\frac{(2-n) x_i}{
\big ( |\tilde x|^2+(x_n+\mathfrak D)^2-1
\big ) ^{\frac{n}{2}}},\quad i=1, \ldots, n-1
\end{equation}
and
\begin{equation}\label{Jn}
\begin{split}
\mathfrak j_n(x)&:=
\Big ( \frac{2-n}{2}U(x)-\nabla U(x)\cdot 
(x+\mathfrak D\mathfrak e_n)+\mathfrak D
\frac{\partial U}{\partial x_n}
\Big ) \\
&=\frac{\alpha_n}{|\mathbf  K(p)|^{\frac{n-2}{4}}}\frac{n-2}{2}
\frac{|x|^2+1-\mathfrak D^2}{
\big ( |\tilde x|^2+(x_n+\mathfrak D)^2-1
\big ) ^{\frac{n}{2}}}.
\end{split}
\end{equation}

As it happens in the Yamabe problem, bubbles are not a good enough approximating solution due to some error terms arising from the geometry of
$M$. Therefore, they need to be corrected by a higher order term
$V_p:\mathbb R^n_+\to \mathbb R$, whose main properties are collected in the next proposition.\\

\begin{proposition}\label{vp} Let
$U$ be as in \eqref{U}. We set
	$$
		\mathtt E_p(x)=c_n \left(\frac 13 \bar R_{ijk\ell}(p) y_k y_\ell +R_{ninj}(p) y_n^2\right)\de^2_{ij} U(x),\ x\in \mathbb R^n_+
	$$
Then the problem
\begin{equation}\label{LEp}
\begin{cases}
\displaystyle
		-\frac{4(n-1)}{n-2}\Delta V + \frac{n+2}{n-2}\abs{\mathbf  K(p)}U^\frac{4}{n-2}V=	\mathtt E_p & \text{ in } \R^n_+, \\[6pt]
		\displaystyle
		\frac{2}{n-2}\frac{\de V}{\de \nu} - \frac{n}{n-2}\mathbf  H(p) U^\frac{2}{n-2}V=0 & \text{ on } \de \R^n_+,
\end{cases}
\end{equation}
admits a solution
$V_p$ satisfying the following properties:
\begin{enumerate}
	\item[(i)]
$\displaystyle\int\limits_{\R^n_+}V_p(x)\:\mathfrak j_i(x)dx=0$ for any
$ i=1,\ldots,n$ (see \eqref{Ji} and \eqref{Jn}),
	\item[(ii)]
$\abs{\nabla^\alpha V_p}(x) \lesssim 
\big ( 1+\abs{x}
\big ) ^{4-n-\alpha}\,$ for any 
$x\in\R^n_+
$ and
$\alpha = 0,1,2,$
	\item[(iii)]  
	\begin{equation*}
		|\mathbf K(p)|\int_{\R^n_+} U^\frac{n+2}{n-2}V_pdx = (n-1)\mathbf H(p)\int_{\de\R^n_+} U^\frac{n}{n-2}V_p\,d\tilde x,
	\end{equation*}
	\item[(iv)]   We have that
$$\int_{\R^n_+}
\Bigg ( -\frac{4(n-1)}{n-2}\Delta V_p+\frac{n+2}{n-2}\abs{\mathbf  K(p)}U^\frac{4}{n-2}V_p
\Bigg ) V_p \geq 0$$
		\item[(v)] The map
$p\mapsto V_p$ is
$C^2(\de M)$. 
\end{enumerate}
\end{proposition}
\begin{proof}
Let $\bar U=|\mathbf K(p)|^{\frac{n-2}{4}}U$. Then problem \eqref{LEp} becomes 
\begin{equation}\label{LEp1}
\begin{cases}
\displaystyle
		-\frac{4(n-1)}{n-2}\Delta V + \frac{n+2}{n-2}\bar U^\frac{4}{n-2}V=	\mathtt E_p & \text{ in } \R^n_+, \\[6pt]
		\displaystyle
		\frac{2}{n-2}\frac{\de V}{\de \nu} - \frac{n}{n-2}\frac{\D}{\sqrt{n(n-1)}}\bar U^\frac{2}{n-2}V=0 & \text{ on } \de \R^n_+.
\end{cases}
\end{equation}
Let $\Phi$ the map given by $$\Phi=\mathcal K^{-1}\circ \tau_{\D}:\mathbb R^n_+\to B_1(0)\subset \mathbb R^n$$ where $\tau_{\D}$ is the translation $x\mapsto x+\D \mathfrak e_n$ while $\mathcal K$ is the Cayley transform which maps conformally the ball of radius $1$ centered at the origin of $\mathbb R^n$ to the half-space $\mathbb R^n_+$.\\ It can be proved that, up to composing with a certain isometry of $\mathbb H^n$, the hyperbolic space, $${\rm Im}(\Phi)=B_R(0),\quad R=\D-\sqrt{\D^2-1}.$$ Moreover, $\Phi$ is a conformal map and $$\Phi^*g_{\mathbb H}=\frac{|\mathbf K(p)|}{n(n-1)}U^{\frac{4}{n-2}}g_0,\quad g_{\mathbb H}:=\frac{4|dx|^2}{(1-|x|^2)^2}\quad\mbox{on}\,\, B_R.$$ Then, $\hat v=(\bar U^{-1}v)\circ \Phi^{-1}$ is in $H^1(B_R)$ and satisfies the problem
\begin{equation}\label{LEp2}
\begin{cases}
\displaystyle
		\Delta_{\mathbb H} \hat v -n \hat v=\hat f  & \text{ in } B_R, \\[6pt]
		\displaystyle
	\frac{\de \hat v}{\de \nu_{\mathbb H}} =\D\hat v & \text{ on } \de \partial B_R
\end{cases}
\end{equation}
where we set $$\hat f(\Phi^{-1}(x))=\frac{n(n-2)}{4}\mathtt E_p(x)\bar U^{-\frac{n+2}{n-2}}$$ and where $$\Delta_{\mathbb H}\hat v =\frac{(1-|x|^2)^2}{4}\Delta\hat v+\frac{n-2}{2}\nabla \hat v \cdot x,\qquad \frac{\partial\hat v}{\partial \nu_{\mathbb H}}=\frac{1-|x|^2}{2}\frac{\partial\hat v}{\partial\nu}$$ are the Laplace-Beltrami operator and the normal derivative with respect to $g_{\mathbb H}$ respectively.\\ Now, it is known (see Lemma 2.3 in \cite{Cruz-BlazquezPistoiaVaira_2022}) that the first eigenvalue of the problem
\begin{equation}\label{LEp4}
\begin{cases}
\displaystyle
		\Delta_{\mathbb H} \hat \phi -n \hat \phi=0  & \text{ in } B_R, \\[6pt]
		\displaystyle
	\frac{\de \hat \phi}{\de \nu_{\mathbb H}} =\mu\hat \phi & \text{ on } \de \partial B_R
\end{cases}
\end{equation}
is $$\mu_0:=\frac{2R}{1+R^2}$$ and the corresponding eigenfuction is $$\phi_0:=\frac{1+|x|^2}{1-|x|^2}.$$ The second eigenvalue of \eqref{LEp4} is $$\mu_1:=\frac{1+R^2}{2R}$$ and the corresponding eigenspace is generated by the family $$\left\{\phi_1^i:=\frac{|x|x_i}{1-|x|^2},\quad i=1, \ldots, n\right\}.$$ Moreover it is shown in Lemma 2.3 in \cite{Cruz-BlazquezPistoiaVaira_2022} that $\hat{\mathfrak j}_i:=c_i \phi_i^1$.\\ Now, since $\D= \frac{1+R^2}{2R}$ with $R=\D-\sqrt{\D^2-1}$ then in \eqref{LEp2} we have that $\D$ is the second eigenvalue and a solution of \eqref{LEp2} exists if $\hat f$ is orthogonal to the elements of the kernel.\\ Indeed, we have that by the area formula and the fact that $\hat{\mathfrak j}_i:=c_i \phi_i^1$ we have that
$$\int_{B_R}\hat f (z) \phi_1^s(z)d\mu_{g_{\mathbb H}}=c_n \int_{\mathbb R^n_+}\left(\frac 13 \bar R_{ijk\ell}(p)x_k x_\ell +R_{ninj}(p) x_n^2\right)\partial^2_{ij}U \mathfrak j_s(x)\, dx.$$ Now, if $s=1, \ldots, n-1$ then by symmetry reason (since the integrand is odd with respect to $\tilde x$ that 
$$\int_{\mathbb R^n_+}\left(\frac 13 \bar R_{ijk\ell}(p)x_k x_\ell +R_{ninj}(p) x_n^2\right)\partial^2_{ij}U \mathfrak j_s(x)\, dx=\int_{\mathbb R^n_+}\left(\frac 13 \bar R_{ijk\ell}(p)x_k x_\ell +R_{ninj}(p) x_n^2\right)\partial^2_{ij}U \partial_s U(x)\, dx=0.$$
For $s=n$ we have that $$\mathfrak j_n(x):=\frac{2-n}{2}U(x)-\nabla U\cdot (x+\D\mathfrak e_n)+\D\frac{\partial U}{\partial x_n}=\frac{2-n}{2}U(x)-\sum_{a=1}^n x_a\partial_a U.$$
We also remark that for $i\neq j$ $$\partial^2_{ij}U(x)=\frac{\alpha_n}{|\mathbf K(p)|^{\frac{n-2}{4}}}\frac{n(n-2)x_ix_j}{\left(|\tilde x|^2+(x_n+\D)^2-1\right)^{\frac{n+2}{2}}}$$ while for $i=j$ then $\bar R_{iik\ell}=0$ and $R_{nini}=R_{nn}=0$. Then
$$\begin{aligned}
&\int_{\mathbb R^n_+}\left(\frac 13 \bar R_{ijk\ell}(p)x_k x_\ell +R_{ninj}(p) x_n^2\right)\partial^2_{ij}U U(x)\, dx=\\
&=\frac{\alpha_n^2}{|\mathbf K(p)|^{\frac{n-2}{2}}}\sum_{i\neq j}\int_{\mathbb R^n_+}\left(\frac 13 \bar R_{ijk\ell}(p)x_k x_\ell +R_{ninj}(p) x_n^2\right)\frac{n(n-2)x_ix_j}{\left(|\tilde x|^2+(x_n+\D)^2-1\right)^{n}}\, dx\\
&=\frac{\alpha_n^2}{|\mathbf K(p)|^{\frac{n-2}{2}}}\sum_{k}\int_{\mathbb R^n_+}\left(\frac 13 \bar R_{k\ell k\ell}(p)+\frac 13\bar R_{\ell k k\ell}(p)\right)\frac{n(n-2)x_k^2x_\ell^2}{\left(|\tilde x|^2+(x_n+\D)^2-1\right)^{n}}\, dx\\
&=0
\end{aligned}$$ The first equality is due to the fact that when $i=j$ all the terms are zero since $\bar R_{iik\ell}=0$ and $R_{nini}=R_{nn}=0$.\\ The second equality is due to the fact that when $i\neq j$ then by symmetry all the terms with $x_n^2 x_ix_j$ vanish and the other terms are non zero only when $i=k$ and $j=\ell$ or when $j=k$ and $i=\ell$. The last equality is due to the fact that $\bar R_{k\ell k\ell}=-\bar R_{\ell k k\ell}$.\\ Now, as before,
$$\begin{aligned}
&\sum_{a=1}^n\int_{\mathbb R^n_+}\left(\frac 13 \bar R_{ijk\ell}(p)x_k x_\ell +R_{ninj}(p) x_n^2\right)\partial^2_{ij}U x_a\partial_a U\, dx=\\
&=\frac{\alpha_n}{|\mathbf K(p)|^{\frac{n-2}{4}}}\sum_{i\neq j}\int_{\mathbb R^n_+}\left(\frac 13 \bar R_{ijk\ell}(p)x_k x_\ell +R_{ninj}(p) x_n^2\right)\frac{n(n-2)x_ix_j}{\left(|\tilde x|^2+(x_n+\D)^2-1\right)^{\frac{n+2}{2}}}\sum_{a=1}^n x_a\partial_a U\, dx\\
&=-\frac{\alpha_n^2}{|\mathbf K(p)|^{\frac{n-2}{2}}}n(n-2)^2\sum_{k}\int_{\mathbb R^n_+}\left(\frac 13 \bar R_{k\ell k\ell}(p)+\frac 13\bar R_{\ell k k\ell}(p)\right)\frac{x_k^2x_\ell^2\left(\sum_{a=1}^{n-1}x_a^2+x_n(x_n+\D)\right)}{\left(|\tilde x|^2+(x_n+\D)^2-1\right)^{n+1}}\, dx\\
&=0.
\end{aligned}$$ Hence  $$\int_{B_R}\hat f (z) \phi_1^s(z)d\mu_{g_{\mathbb H}}=0$$ and, by elliptic linear theory, there exists a solution $\hat v$ to \eqref{LEp2} which is orthogonal to $\left\{\phi_1^s\right\}_{s=1, \ldots, n}$. Consequently $v=\bar U (\hat v\circ \Phi)$ is a solution of \eqref{LEp1} which is orthogonal to $\mathfrak j_s$ with $s=1, \ldots, n$.\\\\ In order to show (ii)-(iii)-(iv)-(v) one can reason exactly as in the proof of Proposition 3.1 of \cite{Cruz-BlazquezPistoiaVaira_2022}.
\end{proof}
\begin{remark} It is hard to show (but it is reasonable that it is like this) that there exists a positive constant
$\mathtt F_n$ such that
$$\int_{\R^n_+}
\Bigg ( -\frac{4(n-1)}{n-2}\Delta V_p+\frac{n+2}{n-2}\abs{\mathbf K}U^\frac{4}{n-2}V_p
\Bigg ) V_p = \mathtt F_n -2R^2_{nins}\mathtt S$$ where $\mathtt S$ is defined in \eqref{s}. This means that the reduced functional given in Lemma \eqref{redf} when $\mathbf K, \mathbf H$ are constants is given by 
$$J_{\e, g}(\tilde\Theta_{\delta, g})=\mathfrak E+\mathtt A\gamma(p)\e\delta  -\delta^4\mathtt{\tilde F}_n|\overline{{\rm Weyl}_g}(p)|^2  +\mathcal O(\delta^5)$$ and hence the location of the blow-up depends on the non-degenerate critical points of the Weyl tensor on the boundary of the manifold. \end{remark}

We are now able to define the good ansatz 
of the solution we are looking for. To this aim, take
$p\in\partial M$ with
$\D>1$ and consider
$\psi_p^\partial: \mathbb R^n_+\to M$ 
the Fermi coordinates in a neighborhood of
$p$. Then, let us define 
\begin{equation*} 
\mathcal W_{p,\delta}(\xi) :=\chi
\Big ( 
  ( \psi_{p}^\partial
  ) ^{-1}(\xi)
\Big ) 
\frac{1}{\delta^{\frac{n-2}{2}}}U
\Big ( \frac{
 ( \psi_{p}^\partial
  ) ^{-1}(\xi)}{\delta}
\Big ) , \quad \mathcal V_{\delta, p}(\xi) :=\chi
\Big ( 
  ( \psi_{p}^\partial
  ) ^{-1}(\xi)
\Big ) 
\frac{1}{\delta^{\frac{n-2}{2}}}V_p
\Big ( \frac{
 ( \psi_{p}^\partial
  ) ^{-1}(\xi)}{\delta}
\Big )
\end{equation*}
where
$\chi$ is a radial cut-off function with
 support in a ball of radius
$R$. \\ Moreover, for any
$i=1, \ldots, n$,
we also set
\begin{equation*}
	\mathcal Z_{\delta, p, i}(\xi):=\frac{1}
	{\delta^{\frac{n-2}{2}}}\mathfrak j_i
\Big ( \frac{
  ( \psi_p^\partial
  ) ^{-1}(\xi)}{\delta}
\Big ) \chi
\Big ( 
  ( \psi_p^\partial
  ) ^{-1}(\xi)
\Big ) ,
\end{equation*} 
being
$\mathfrak j_i$ the functions defined in \eqref{Ji} 
and \eqref{Jn}.\\\\
 We look for a solution of the form $$u_\varepsilon=\mathcal W_{\delta, p}+\delta^2\mathcal V_{\delta, p}+\Phi$$ where $\Phi$ is a remainder term. So we will find a solution of the original problem \eqref{pb} of the form \begin{equation}\label{vsol}v_\varepsilon:=\Lambda_p\left(\mathcal W_{\delta, p}+\delta^2\mathcal V_{\delta, p}+\Phi\right).\end{equation}
In the following we simply use $\tilde{\mathcal W}_{\delta, p}, \tilde{\mathcal V}_{\delta, p}, \tilde\Phi, \tilde{\mathcal Z}_{\delta, p, i}$ to denote $\Lambda_p\mathcal W_{\delta, p}, \Lambda_p\mathcal V_{\delta, p}, \Lambda_p\Phi, \Lambda_p\mathcal Z_{\delta, p, i}$ respectively.\\ In order to simplify the notation we let $$\Theta_{\delta, p}:=\mathcal W_{\delta, p}+\delta^2\mathcal V_{\delta, p}\quad\mbox{and}\quad\tilde {\Theta}_{\delta, p}:=\tilde{\mathcal W}_{\delta, p}+\delta^2 \tilde{\mathcal V}_{\delta, p}.$$

Let us decompose
$H^1(M)$ into the direct sum of the 
following two subspaces
$$
\tilde{\mathcal K}:={\rm span}
\big\{\tilde{\mathcal Z}_{\delta, p, i}\,\,:\,\, i=1,
 \ldots, n\,\, \big\}
 $$ 
 and 
 $$
\tilde{ \mathcal K}^\bot:=
 \big\{\varphi\in H^1(M)\,:\, \langle
  \varphi, \tilde{\mathcal Z}_{\delta, p, i}\rangle_g=0,\quad i=1,
   \ldots, n\,\, \big\}.\\
$$ 

In order to solve \eqref{pb1} we will 
use the following finite-dimensional 
reduction: define the projections
$$
\Pi: H^1(M)\to \tilde{ \mathcal K},\qquad 
\Pi^\bot: H^1(M)\to \tilde{\mathcal K}^\bot.
$$
Therefore,
 solving \eqref{pb1} 
is equivalent to solve the system
\begin{equation}\label{aux}
\Pi^\bot
\Big \{v_\varepsilon
 -\mathfrak i^*_M(\mathbf K \mathfrak g(v_\varepsilon))-
 \mathfrak i^*_{\partial M}
\Big ( \frac{n-2}{2}
  ( \mathbf H \mathfrak f(v_\varepsilon)
-\varepsilon \gamma v_\varepsilon
 ) 
\Big ) \Big\}=0, 
\end{equation}
\begin{equation}
\label{bif} 
\Pi 
\Big \{v_\varepsilon 
-\mathfrak i^*_M(\mathbf K \mathfrak g(v_\varepsilon))
-\mathfrak i^*_{\partial M}
\Big ( \frac{n-2}{2}
  ( \mathbf H \mathfrak f(v_\varepsilon)-\varepsilon \gamma
v_\varepsilon
  ) 
\Big ) \Big\}=0,
\end{equation}  
with
$v_\varepsilon$ defined in \eqref{vsol}.

\eqref{aux} and \eqref{bif} are called the 
\textit{auxiliar} and the \textit{bifurcation} 
equations, respectively. Solving \eqref{aux} 
will give us the error term
$\tilde\Phi$. With the size of the error in mind, 
we check that our choice of the parameters 
leads to a solution of \eqref{bif}.

\section{the finite dimensional reduction}\label{riduzione}
The equation \eqref{aux} is equivalent to
\begin{equation}\label{LNE}
L(\tilde\Phi)=N(\tilde\Phi)+R
\end{equation}
where $L:\tilde{\mathcal K}^\bot\to \tilde{\mathcal K}^\bot$ is the linear operator defined as
\begin{equation}\label{L}
L(\tilde\Phi):=\Pi^\bot
\Big \{\tilde\Phi
 -\mathfrak i^*_M(\mathbf K \mathfrak g'(\tilde {\Theta}_{\delta, p})\tilde\Phi)-
 \mathfrak i^*_{\partial M}
\Big ( \frac{n-2}{2}
  \left( \mathbf H \mathfrak f'(\tilde {\Theta}_{\delta, p})\tilde\Phi
-\varepsilon \gamma\tilde\Phi
 \right) 
\Big ) \Big\}, 
\end{equation}
while the nonlinear term $N(\tilde\Phi)$ and the error $R$ are given respectively 
\begin{equation}\label{N}\begin{aligned}
N(\tilde \Phi)&:=\Pi^\bot
\Big \{
 \mathfrak i^*_M\left(\mathbf K \left(\mathfrak g(\tilde {\Theta}_{\delta, p}+\tilde\Phi)-\mathfrak g(\tilde {\Theta}_{\delta, p})-\mathfrak g'(\tilde{\Theta}_{\delta, p})\tilde\Phi\right)\right)\Big\}+\\
&+\Pi^\bot
\Big \{
 \mathfrak i^*_{\partial M}
\Big ( \frac{n-2}{2}
  \left( \mathbf H \left(\mathfrak f(\tilde {\Theta}_{\delta, p}+\tilde\Phi)-\mathfrak f(\tilde {\Theta}_{\delta, p})-\mathfrak f'(\tilde{\Theta}_{\delta, p})\tilde\Phi\right)\right)
\Big\}\end{aligned}
\end{equation}
\begin{equation}\label{R}
R:=\Pi^\bot
\Big \{
 \mathfrak i^*_M(\mathbf K \mathfrak g(\tilde {\Theta}_{\delta, p}))+
 \mathfrak i^*_{\partial M}
\Big ( \frac{n-2}{2}
  \left( \mathbf H \mathfrak f(\tilde {\Theta}_{\delta, p})
-\varepsilon \gamma\tilde {\Theta}_{\delta, p}
 \right) 
\Big )-\tilde {\Theta}_{\delta, p} \Big\}
\end{equation}

\subsection{The size of the error}
\begin{lemma}\label{errore}Assume $n\geq 8$ then it holds 
\begin{equation}\label{errorestima}
\|R\|_g \lesssim \left\{\begin{aligned} &\delta^2+\varepsilon\delta \quad &\mbox{if}\,\, &\mathbf K\, \mbox{and}\, \mathbf H \, \hbox{are not constants}\\
&\delta^3+\e \delta \quad &\mbox{if}\,\, &\mathbf K\, \mbox{and}\, \mathbf H \, \hbox{are constants}\end{aligned}\right.\end{equation}
\end{lemma}
\begin{proof}
Let $$\gamma_M:=\mathfrak i^*_M(\mathbf K \mathfrak g(\tilde {\Theta}_{\delta, p}))\quad\hbox{and}\quad \gamma_{\partial M}:=\mathfrak i^*_{\partial M}
\Big ( \frac{n-2}{2}
  \left( \mathbf H \mathfrak f(\tilde {\Theta}_{\delta, p})
-\varepsilon \gamma\tilde {\Theta}_{\delta, p}
 \right) 
\Big ).$$ Hence it follows that
\begin{equation}\label{ii*}
\left\{\begin{aligned}&\mathcal L_g \gamma_M=\mathbf K \mathfrak g(\tilde {\Theta}_{\delta, p})\, &\hbox{in}\, M\\
&\mathcal B_g\gamma_M=0\, &\hbox{on}\, \partial M\end{aligned}\right.\quad
\left\{\begin{aligned}&\mathcal L_g\gamma_{\partial M}=0\quad &\hbox{in}\, M\\
&\mathcal B_g\gamma_{\partial M}=\frac{n-2}{2}
  \left( \mathbf H \mathfrak f(\tilde {\Theta}_{\delta, p})
-\varepsilon \gamma\tilde {\Theta}_{\delta, p}
 \right) \, &\hbox{on}\, \partial M\end{aligned}\right.
\end{equation}
Since $$d\nu_{\tilde g_p}=\sqrt{{\rm det}\tilde g_p}\, dx= \sqrt{{\rm det}(\Lambda_p^{\frac{4}{n-2}}g)}\, dx=\Lambda^{\frac{2n}{n-2}}\sqrt{{\rm det}g}\, dx=\Lambda_p^{2^*}d\nu_g$$
and similarly
$$d\sigma_{\tilde g_p}=\Lambda_p^{2^\sharp}d\sigma_g$$ we get that
$$\begin{aligned}\|R\|_g^2&=\|\gamma_M+\gamma_{\partial M}-\tilde {\Theta}_{\delta, p}\|^2_g\\
&=\int_M \mathcal L_g(\gamma_M+\gamma_{\partial M}-\tilde {\Theta}_{\delta, p})(\gamma_M+\gamma_{\partial M}-\tilde {\Theta}_{\delta, p})\, d\nu_g\\
&+\int_{\partial M}c_n\mathcal B_g (\gamma_M+\gamma_{\partial M}-\tilde {\Theta}_{\delta, p})(\gamma_M+\gamma_{\partial M}-\tilde {\Theta}_{\delta, p})\,d\sigma_g\\
&=\int_M \left(-\mathcal L_g(\tilde {\Theta}_{\delta, p})+\mathbf K\mathfrak g(\tilde {\Theta}_{\delta, p})\right) R\, d\nu_g\\
&+c_n\int_{\partial M}\left(-\mathcal B_g (\tilde {\Theta}_{\delta, p})+\mathcal B_g(\gamma_{\partial M})\right)R\, d\sigma_g\\
&=\int_M \left(-\mathcal L_{\tilde g_p}( {\Theta}_{\delta, p})+\mathbf K\mathfrak g({\Theta}_{\delta, p})\right) \Lambda_p^{-1}R\, d\nu_{\tilde g_p}\\
&+c_n\int_{\partial M}\left(-\mathcal B_{\tilde g_p} ({\Theta}_{\delta, p})+\mathcal B_{\tilde g_p}(\Lambda_p^{-1}\gamma_{\partial M})\right)\Lambda_p^{-1}R\, d\sigma_{\tilde g_p}\\
&=\int_M \left(c_n\Delta_{\tilde g_p}{\Theta}_{\delta, p}-k_{\tilde g_p} {\Theta}_{\delta, p}+\mathbf K\mathfrak g({\Theta}_{\delta, p})\right) \Lambda_p^{-1}R\, d\nu_{\tilde g_p}\\
&+c_n\int_{\partial M}\left(-\frac{\partial {\Theta}_{\delta, p}}{\partial \nu}-\frac{n-2}{2}h_{\tilde g_p}{\Theta}_{\delta, p}+\frac{n-2}{2}\mathbf H\mathfrak f({\Theta}_{\delta, p})-\frac{n-2}{2}\varepsilon\Lambda_p^{-\frac{2}{n-2}}\gamma{\Theta}_{\delta, p}\right)\Lambda_p^{-1}R\, d\sigma_{\tilde g_p}\\
&\lesssim \|c_n\Delta_{\tilde g_p}{\Theta}_{\delta, p}+\mathbf K\mathfrak g({\Theta}_{\delta, p})\|_{L^\frac{2n}{n+2}(M, \tilde g_p)}\|\Lambda_p^{-1}R\|_{\tilde g_p}+\|k_{\tilde g_p} {\Theta}_{\delta, p}\|_{L^\frac{2n}{n+2}(M, \tilde g_p)}\|\Lambda_p^{-1}R\|_{\tilde g_p}\\
&+\left\|-\frac{\partial {\Theta}_{\delta, p}}{\partial \nu}+\frac{n-2}{2}\mathbf H\mathfrak f({\Theta}_{\delta, p})\right\|_{L^\frac{2(n-1)}{n}(\partial M, \tilde g_p)}\|\Lambda_p^{-1}R\|_{\tilde g_p}\\
&+\|h_{\tilde g_p}{\Theta}_{\delta, p}\|_{L^\frac{2(n-1)}{n}(\partial M, \tilde g_p)}\|\Lambda_p^{-1}R\|_{\tilde g_p}+\|\varepsilon\Lambda_p^{-\frac{2}{n-2}}\gamma{\Theta}_{\delta, p}\|_{L^\frac{2(n-1)}{n}(\partial M, \tilde g_p)}\|\Lambda_p^{-1}R\|_{\tilde g_p}\\
\end{aligned}$$
Now, we recall the expression for the Laplace-Beltrami operator in local charts
$$\begin{aligned}\Delta_{\tilde g_p}&:=\Delta +\left[\tilde{g}^{ij}_p(y)-\delta_{ij}\right]\partial^2_{ij}+\left[\partial_i \tilde{g}^{ij}_p(y)+\frac{\tilde{g}^{ij}(y)\partial_i|\tilde g_p|^{\frac 12}(y)}{|\tilde g_p|^{\frac 12}(y)}\right]\partial_j +\frac{\partial_n|\tilde g_p|^{\frac 12}(y)}{|\tilde g_p|^{\frac 12}(y)}\partial_n.\end{aligned}$$
Then, in variables $y=\delta x$
$$\begin{aligned} c_n\Delta_{\tilde g_p}W_{\delta, p}&=\delta^{-\frac{n-2}{2}-2}c_n\Delta(U(x)\chi(\delta x))+\delta^{-\frac{n-2}{2}}c_n \left(\frac 13 \bar R_{ikj\ell}x_kx_\ell+R_{ninj}x_n^2+\delta\mathcal O(|x|^3)\right)\partial^2_{ij}(U(x)\chi(\delta x))\\
&+\delta^{-\frac{n-2}{2}}\frac{c_n}{3}\left(\bar R_{iij\ell} x_\ell+\bar R_{ikji}x_k+\delta\mathcal O(|x|^2)\right)\partial_j(U(x)\chi(\delta x))\\
&+\delta^{-\frac{n-2}{2}}\delta^2\mathcal O(|x|^3)\partial_n(U(x)\chi(\delta x)).\\
&=\delta^{-\frac{n+2}{2}}\underbrace{\left(|\mathbf K(p)|U^{2^*-1}\right)}_{\tiny{by \eqref{pblim}}}\chi(\delta x)+\underbrace{\left[\delta^{-\frac{n-2}{2}}c_n\left(\frac 13 \bar R_{ikj\ell}x_kx_\ell+R_{ninj}x_n^2\right)\partial^2_{ij}U(x)\right]}_{:=\delta^{-\frac{n-2}{2}}\mathtt E_p(x)}\chi(\delta x)\\
&+\mathcal O(\delta^{-\frac{n-2}{2}+1}|x|^3\partial^2_{ij}U(x))+\mathcal O(\delta^{-\frac{n-2}{2}+1}|x|^2\partial_j U(x))
\end{aligned}$$
We remark that, by symmetry, $\bar R_{iij\ell}=0$ and also $\bar R_{ikji}=-\bar
R_{jk}=0$ (see \cite{Marques}). 
Now 
$$\begin{aligned} \delta^2c_n\Delta_{\tilde g_p}V_{\delta, p}&=\delta^{-\frac{n-2}{2}}\Delta (V_p(x)\chi(\delta x))+\delta^{-\frac{n-2}{2}+2}\left(\mathcal O(|x|^2)\partial^2_{ij}(V_p(x)\chi(\delta x))+\mathcal O(|x|)\partial_j (V_p(x)\chi(\delta x))\right)\\
&=\underbrace{\delta^{-\frac{n-2}{2}}\left[|\mathbf K(p)|\mathfrak g'(U)V_p-\mathtt E_p(x)\right]}_{\tiny{by \eqref{LEp}}}\chi(\delta x)\\&+\delta^{-\frac{n-2}{2}+2}\left(\mathcal O(|x|^2)\partial^2_{ij}(V_p(x)\chi(\delta x))+\mathcal O(|x|)\partial_j (V_p(x)\chi(\delta x))\right).
\end{aligned}$$
Hence, in variables $y=\delta x$
$$\begin{aligned}
-\Delta_{\tilde g_p}( {\Theta}_{\delta, p})+\mathbf K\mathfrak g({\Theta}_{\delta, p})&=c_n\Delta_{\tilde g_p}W_{\delta, p}+\delta^2c_n\Delta_{\tilde g_p}V_{\delta, p}+\mathbf K\mathfrak g(W_{\delta, p}+\delta^2 V_{\delta, p})\\
&\hspace{-2cm }=\delta^{-\frac{n+2}{2}}\mathbf K(\delta x) \left(\mathfrak g(U+\delta^2 V_p)-\mathfrak g(U)-\delta^2 \mathfrak g'(U)V_p\right)(\chi(\delta x))^{2^*-1}\\
&\hspace{-2cm }+\delta^{-\frac{n+2}{2}}\left(\mathbf K(\delta x)-\mathbf K(p)\right)\mathfrak g(U)(\chi(\delta x))^{2^*-1}+\delta^{-\frac{n+2}{2}}\mathbf K(p)\mathfrak g(U)\left((\chi(\delta x)^{2^*-1}-\chi(\delta x)\right)\\
&\hspace{-2cm }+\delta^{-\frac{n-2}{2}}\left(\mathbf K(\delta x)-\mathbf K(p)\right)\mathfrak g'(U)V_p(\chi(\delta x))^{2^*-1}+\delta^{-\frac{n-2}{2}}\mathbf K(p)\mathfrak g'(U)V_p\left((\chi(\delta x)^{2^*-1}-\chi(\delta x)\right)\\
&\hspace{-2cm }+\mathcal O(\delta^{-\frac{n-2}{2}+1}|x|^3\partial^2_{ij}U(x))+\mathcal O(\delta^{-\frac{n-2}{2}+1}|x|^2\partial_j U(x))\\
&\hspace{-2cm }+\delta^{-\frac{n-2}{2}+2}\left(\mathcal O(|x|^2)\partial^2_{ij}(V_p(x)\chi(\delta x))+\mathcal O(|x|)\partial_j (V_p(x)\chi(\delta x))\right).
\end{aligned}$$
Now
$$\left|\mathfrak g(U+\delta^2 V_p)-\mathfrak g(U)-\delta^2 \mathfrak g'(U)V_p\right|
\lesssim U^{\frac{6-n}{n-2}}(\delta^2 V_p)^2\quad \hbox{since}\, \,n\geq 7
$$
Hence we get
$$\|\delta^{-\frac{n+2}{2}}\mathbf K(\delta x) \left(\mathfrak g(U+\delta^2 V_p)-\mathfrak g(U)-\delta^2 \mathfrak g'(U)V_p\right)(\chi(\delta x))^{2^*-1}\|_{L^{\frac{2n}{n+2}}(M, \tilde g_p)}\lesssim \delta^4. $$
Now if $\mathbf K$ is constant then $\delta^{-\frac{n+2}{2}}\left(\mathbf K(\delta x)-\mathbf K(p)\right)\mathfrak g(U)(\chi(\delta x))^{2^*-1}=0$, while if $\mathbf K$ is not a constant, since $p$ is a non-degenerate critical point of $\mathbf K$ then we get
$$\|\delta^{-\frac{n+2}{2}}\left(\mathbf K(\delta x)-\mathbf K(p)\right)\mathfrak g(U)(\chi(\delta x))^{2^*-1}\|_{L^{\frac{2n}{n+2}}(M, \tilde g_p)}\lesssim \delta^2.$$ Moreover
$$\|\delta^{-\frac{n+2}{2}}\mathbf K(p)\mathfrak g(U)\left((\chi(\delta x))^{2^*-1}-\chi(\delta x)\right)\|_{L^{\frac{2n}{n+2}}(M, \tilde g_p)}\lesssim \delta^3.$$
Again, if $K$ is constant then $\delta^{-\frac{n-2}{2}}\left(\mathbf K(\delta x)-\mathbf K(p)\right)\mathfrak g'(U)V_p(\chi(\delta x))^{2^*-1}=0$ while, if $\mathbf K$ is not constant we get
$$\|\delta^{-\frac{n-2}{2}}\left(\mathbf K(\delta x)-\mathbf K(p)\right)\mathfrak g'(U)V_p(\chi(\delta x))^{2^*-1}\|_{L^{\frac{2n}{n+2}}(M, \tilde g_p)}\lesssim \delta^4$$
and $$\|\delta^{-\frac{n-2}{2}}\mathbf K(p)\mathfrak g'(U)V_p\left((\chi(\delta x)^{2^*-1}-\chi(\delta x)\right)\|_{L^{\frac{2n}{n+2}}(M, \tilde g_p)}\lesssim \delta^4.$$
Hence
\begin{equation}\label{lapl}
\|c_n\Delta_{\tilde g_p}{\Theta}_{\delta, p}+\mathbf K\mathfrak g({\Theta}_{\delta, p})\|_{L^\frac{2n}{n+2}(M, \tilde g_p)}\lesssim\left\{\begin{aligned} &\delta^3 \quad &\hbox{if}\,\, \mathbf K\,\, \hbox{is constant}\\
&\delta^2 \quad &\hbox{if}\,\, \mathbf K\,\, \hbox{is not a constant}\\
\end{aligned}\right.\end{equation}
Now since $k_{\tilde g_p} (p)=k_{\tilde g_p, i}(p)=k_{\tilde g_p, n}(p)=0$ (see \cite{Marques}) then  $$\begin{aligned}\|k_{\tilde g_p} {\Theta}_{\delta, p}\|_{L^\frac{2n}{n+2}(M, \tilde g_p)}&\lesssim \delta^2\left[\int_{\mathbb R^n_+}\left(k_{\tilde g_p} (\delta x) U(x) \chi(\delta x)\right)^{\frac{2n}{n+2}}\right]^{\frac{n+2}{2n}} +\delta^4 \left[\int_{\mathbb R^n_+}\left(k_{\tilde g_p} (\delta x) V_p(x) \chi(\delta x)\right)^{\frac{2n}{n+2}}\right]^{\frac{n+2}{2n}}\\
&\lesssim \delta^4 \left[\int_{\mathbb R^n_+}\left(|x|^2 U(x)\right)^{\frac{2n}{n+2}}\right]^{\frac{n+2}{2n}}+\delta^6 \left[\int_{\mathbb R^n_+}\left( |x|^2 V_p(x) \chi(\delta x)\right)^{\frac{2n}{n+2}}\right]^{\frac{n+2}{2n}}\\
&\lesssim \left\{\begin{aligned} & \delta^4 \quad &\hbox{if}\,\, n\geq 10\\
&\delta^3 \quad &\hbox{if}\,\, n=8, 9.\\
\end{aligned}\right.
\end{aligned}$$
Now
$$\begin{aligned}\|\varepsilon\Lambda_p^{-\frac{2}{n-2}}\gamma{\Theta}_{\delta, p}\|_{L^\frac{2(n-1)}{n}(\partial M, \tilde g_p)}&\lesssim \varepsilon \delta \left(\int_{\mathbb R^{n-1}}|U(\tilde x, 0)\chi(\delta \tilde x, 0)|^{\frac{2(n-1)}{n}}\, dx\right)^{\frac{n}{2(n-1)}}\\
&+\varepsilon \delta^3 \left(\int_{\mathbb R^{n-1}}|V_p(\tilde x, 0)\chi(\delta \tilde x, 0)|^{\frac{2(n-1)}{n}}\, dx\right)^{\frac{n}{2(n-1)}}\\
&\lesssim \varepsilon \delta\end{aligned}$$ since for $n\geq 5$ $$\int_{\mathbb R^{n-1}}|U(\tilde x, 0)\chi(\delta \tilde x, 0)|^{\frac{2(n-1)}{n}}\, dx<+\infty$$ while 
$$\left(\int_{\mathbb R^{n-1}}|V_p(\tilde x, 0)\chi(\delta \tilde x, 0)|^{\frac{2(n-1)}{n}}\, dx\right)^{\frac{n}{2(n-1)}}\lesssim\left\{\begin{aligned} & c \quad &\hbox{if}\, n\geq 9\\
&|\log\delta|^{\frac 4 7} \quad &\hbox{if}\, n=8.\end{aligned}\right.$$
Since $h_{\tilde g_p}(p)=h_{\tilde g_p, i}(p)=h_{\tilde g_p, ik}(p)=0$ 
\begin{equation}\label{hg}\begin{aligned}\|h_{\tilde g_p}{\Theta}_{\delta, p}\|_{L^\frac{2(n-1)}{n}(\partial M, \tilde g_p)}&\lesssim  \delta^4 \left(\int_{\mathbb R^{n-1}}||\tilde x|^3U(\tilde x, 0)\chi(\delta \tilde x, 0)|^{\frac{2(n-1)}{n}}\, dx\right)^{\frac{n}{2(n-1)}}\\
&+\delta^6  \left(\int_{\mathbb R^{n-1}}||\tilde x|^3V_p(\tilde x, 0)\chi(\delta \tilde x, 0)|^{\frac{2(n-1)}{n}}\, dx\right)^{\frac{n}{2(n-1)}}\\
&\lesssim \delta^3\end{aligned}\end{equation}
since for $n\geq 8$
$$\left(\int_{\mathbb R^{n-1}}||\tilde x|^3U(\tilde x, 0)\chi(\delta \tilde x, 0)|^{\frac{2(n-1)}{n}}\, dx\right)^{\frac{n}{2(n-1)}}\lesssim c $$
and
$$\left(\int_{\mathbb R^{n-1}}||\tilde x|^3V_p(\tilde x, 0)\chi(\delta \tilde x, 0)|^{\frac{2(n-1)}{n}}\, dx\right)^{\frac{n}{2(n-1)}}\lesssim\left\{\begin{aligned} & c \quad &\hbox{if}\, n\geq 12\\
&\delta^{-\frac{3}{20}} \quad &\hbox{if}\, n=11\\
&\delta^{-\frac{2}{3}} \quad &\hbox{if}\, n=10\\
&\delta^{-\frac{19}{16}} \quad &\hbox{if}\, n=9\\
&\delta^{-\frac{12}{7}} \quad &\hbox{if}\, n=8.\end{aligned}\right.$$
At the end
$$\begin{aligned}&-\frac{\partial {\Theta}_{\delta, p}}{\partial \nu}+\frac{n-2}{2}H\mathfrak f({\Theta}_{\delta, p})=-\frac{\partial W_{\delta, p}}{\partial\nu}+\frac{n-2}{2}\mathbf H(p)W_{\delta, p}^{\frac {n}{n-2}}+\frac{n-2}{2}\left(\mathbf H-\mathbf H(p)\right)W_{\delta, p}^{\frac {n}{n-2}}\\
&+\frac{n-2}{2}\mathbf H\left[\mathfrak f(W_{\delta, p}+\delta^2 V_{\delta, p})-\mathfrak f(W_{\delta, p})\right]-\delta^2\frac{\partial V_{\delta, p}}{\partial\nu}\end{aligned}$$
Since $U$ satisfies \eqref{pblim} then $$\left\|-\frac{\partial W_{\delta, p}}{\partial\nu}+\frac{n-2}{2}\mathbf H(p)W_{\delta, p}^{\frac {n}{n-2}}\right\|_{L^\frac{2(n-1)}{n}(\partial M, \tilde g_p)}\lesssim \delta^3$$ 
Now if $\mathbf H$ is constant then $\mathbf H-\mathbf H(p)=0$ while, letting $p$ a non-degenerate critical point of $\mathbf H$ we get that
$$\|\left(\mathbf H-\mathbf H(p)\right)W_{\delta, p}^{\frac {n}{n-2}}\|_{L^\frac{2(n-1)}{n}(\partial M, \tilde g_p)}\lesssim \delta^2.$$ 
At the end we get
$$\begin{aligned}&\hskip-2.0cm \left\|\frac{n-2}{2}\mathbf H\left[\mathfrak f(W_{\delta, p}+\delta^2 V_{\delta, p})-\mathfrak f(W_{\delta, p})\right]-\delta^2\frac{\partial V_{\delta, p}}{\partial\nu}\right\|_{L^\frac{2(n-1)}{n}(\partial M, \tilde g_p)}\\
&\lesssim \left\|\frac{n-2}{2}\mathbf H\left(\mathfrak f(U+\delta^2 V_p)-\mathfrak f(U)\right)\chi^{\frac{n}{n-2}}(\delta\tilde x, 0)-\delta^2\frac{\partial V_p}{\partial\nu}\chi(\delta \tilde x, 0)\right\|_{L^\frac{2(n-1)}{n}(\partial M, \tilde g_p)}\\
&\lesssim \delta^2\left\|\mathbf H\mathfrak f'(U+\delta^2\theta V_p)V_p\chi(\delta\tilde x, 0)-\frac{n}{n-2}\mathbf H U^{\frac{2}{n-2}}\partial V_p\chi(\delta \tilde x, 0)\right\|_{L^\frac{2(n-1)}{n}(\partial M, \tilde g_p)}\\
&\lesssim \delta^2\|\mathbf H(\chi^{\frac{n}{n-2}}-\chi)\mathfrak f'(U+\delta^2\theta V_p)V_p\|_{L^\frac{2(n-1)}{n}(\partial M, \tilde g_p)}\\
&+\delta^2\|\mathbf H(\mathfrak f'(U+\delta^2\theta V_p)-\mathfrak f'(U))V_p\|_{L^\frac{2(n-1)}{n}(\partial M, \tilde g_p)}\\
&\lesssim \delta^3
\end{aligned}$$

\end{proof}

\subsection{Solving the equation \eqref{pb1}}
At this point we can use the same strategy of Proposition 4.1 of \cite{Cruz-BlazquezVaira_2025} to prove the following result
\begin{proposition}\label{phiex}
There exists a positive constant $C$ such that for $\e, \delta$ small, for any $p\in\partial M$ there exists a unique $\tilde\Phi:=\tilde\Phi_{\varepsilon, \delta, p}\in \tilde K^\bot$ which solves \eqref{pb1}such that 
\begin{equation}\label{sizephi}
\|\tilde\Phi\|_g=\|\Lambda_p\Phi\|_g \lesssim \left\{\begin{aligned} &\delta^2+\varepsilon\delta \quad &\mbox{if}\,\, &\mathbf K\, \mbox{and}\, \mathbf H \, \hbox{are not constants}\\
&\delta^3+\e \delta \quad &\mbox{if}\,\, &\mathbf K\, \mbox{and}\, \mathbf H \, \hbox{are constants}\end{aligned}\right.\end{equation}
\end{proposition}
\section{The reduced functional}\label{ridotto}
In this section we perform the expansion of the functional with respect to the parameter $\e$ and $\delta$. First, reasoning as in \cite{Cruz-BlazquezVaira_2025} we get that
\begin{lemma}\label{energiabubble}
The energy of the bubble is:
\begin{align*} 
\mathfrak E(p)
&:=\frac{a_n}{|\mathbf K(p)|^{\frac{n-2}{2}}}
 \Big [-(n-1) \varphi_{\frac{n+1}{2}}(p)+
\frac{\mathfrak D}
{(\mathfrak D^2-1)^{\frac{n-1}{2}}} \Big ],
\end{align*}
where
$$
a_n:=\alpha_n^\dsh \omega_{n-1}I^n_{n-1}\frac{n-3}{(n-1)\sqrt{n(n-1)}}.
$$
\end{lemma}
Moreover, again as in \cite{Cruz-BlazquezVaira_2025} we can show that

$$J_{\e, g}(\tilde W_{\delta, p}+\delta^2V_{\delta, p}+\tilde \Phi)-J_{\e, g}(\tilde W_{\delta, p}+\delta V_{\delta, p})=\left\{\begin{aligned}&\mathcal O\left(\delta^3\right)\quad &\mbox{if}\,\, \mathbf K, \mathbf H \,\,\mbox{are not constants}\\ 
&\mathcal O\left(\delta^5\right)\quad &\mbox{if}\,\, \mathbf K, \mathbf H \,\,\mbox{are constants}\end{aligned}\right.
$$
$C^0-$ uniformly for $p\in\partial M$.
Now we need to expand the energy on the ansatz.

\begin{lemma}\label{redf}
If $\mathbf H, \mathbf K$ are not constants, then, for $\e$ sufficiently small, it holds
$$J_{\e, g}(\tilde\Theta_{\delta, p})=\mathfrak E(p)+\mathtt A(p)\e \delta -\mathtt B(p)\delta^2+\mathcal O(\delta^3)$$
where $\mathfrak E(p)$ is the energy of the bubble evaluated in Lemma \ref{energiabubble} and $\mathtt A(p)$ and $\mathtt B(p)$ are defined in \eqref{ap} and \eqref{bp} respectively.\\ 
If, instead, $\mathbf H, \mathbf K$ are constants, then, for $\e$ sufficiently small, it holds
$$J_{\e, g}(\tilde\Theta_{\delta, g})=\mathfrak E+\mathtt A\gamma(p)\e\delta  -\delta^4\mathtt B(p)+\mathcal O(\delta^5)$$
where $\mathfrak E$ is the energy of the bubble evaluated in Lemma \ref{energiabubble} and $\mathtt A$ and $\mathtt B(p)$ are defined in \eqref{ap} and \eqref{bp1} respectively.\\ 
\end{lemma}
\begin{proof}
We have
$$\begin{aligned}
J_{\e, g}(\tilde \Theta_{\delta, p})&=\tilde J_{\e, \tilde g_p}(\Theta_{\delta, p})=\frac{c_n}{2}\int_M|\nabla_{\tilde g_p}\Theta_{\delta, p}|^2\, d\nu_{\tilde g_p}+\frac 12 \int_M k_{\tilde g_p}\Theta_{\delta, p}^2\, d\nu_{\tilde g_p}+(n-1)\e\int_{\partial M}\Lambda_p^{-\frac{2}{n-2}}\gamma\Theta_{\delta, p}^2\, d\sigma_{\tilde g_p}\\
&-\frac{c_n(n-2)}{2}\int_{\partial M}\mathbf H\left(\mathfrak F(\Theta_{\delta, p})-\mathfrak F(W_{\delta, p})\right)\, d\nu_{\tilde g_p}-\frac{c_n(n-2)}{2}\int_{\partial M} \mathbf H\mathfrak F(W_{\delta, p})\, d\sigma_{\tilde g_p}\\
&+(n-1)\int_{\partial M}h_{\tilde g_p}\Theta_{\delta, p}^2\, d\sigma_{\tilde g_p}-\int_M \mathbf K \left(\mathfrak G(\Theta_{\delta, p})-\mathfrak G(W_{\delta, p})\right)\, d\nu_{\tilde g_p}-\int_M \mathbf K \mathfrak G(W_{\delta, p})\, d\nu_{\tilde g_p}\\
&=A_1+A_2+A_3+A_4+A_5+A_6+A_7+A_8\end{aligned}$$
Now, by \eqref{hij} we get
$$\begin{aligned} A_6&=(n-1)\delta^4 \underbrace{\int_{\mathbb R^{n-1}}\partial^3_{ijk}h_{\tilde g_p}(\tilde x, 0)\tilde x_i\tilde x_j\tilde x_k U^2(\tilde x, 0)\, d\tilde x}_{:=0\, \hbox{\tiny{by symmetry}}} +\mathcal O(\delta^5)=\mathcal O(\delta^5).\end{aligned}$$
By \eqref{Sg} we get
$$\begin{aligned} A_2&=\frac 12 \delta^4\int_{\mathbb R^n_+}\partial^2_{ab}k_{\tilde g_p}(p)x_ax_b U^2(x)\, dx +\mathcal O(\delta^5)\\
&=\frac 14 \delta^4\left[\int_{\mathbb R^n_+}\partial^2_{ii}k_{\tilde g_p}\frac{|\tilde x|^2 U^2(\tilde x, x_n)}{n-1}\, d\tilde x \, d x_n+\partial^2_{nn}k_{\tilde g_p}\int_{\mathbb R^n_+}x_n^2 U^2(\tilde x, x_n)\, d\tilde x\, d x_n\right]+\mathcal O(\delta^5)\\
&=-\frac{1}{24(n-1)}\delta^4|\overline{{\rm Weyl}_g}(p)|^2 \int_{\mathbb R^n_+}|\tilde x|^2 U^2(\tilde x, x_n)\, d\tilde x \, d x_n+\frac{1}{4}\delta^4\partial^2_{nn}k_{\tilde g_p}\int_{\mathbb R^n_+}x_n^2 U^2(\tilde x, x_n)\, d\tilde x\, d x_n+\mathcal O(\delta^5)\\
\end{aligned}$$
Analogously we have, since $\Lambda_p(p)=1$ and $\nabla\Lambda_p(p)=0$ that
$$A_3=(n-1) \e\gamma(p) \delta \int_{\mathbb R^{n-1}}U^2(\tilde x, 0)\, d\tilde x +\mathcal O(\e\delta^3).$$
Now, using the fact that $p$ is a non-degenerate critical point of $\mathbf H$ when $\mathbf H$ is not constant
$$\begin{aligned}A_5&=-\frac{c_n(n-2)}{2}\int_{\mathbb R^{n-1}}\mathbf H(\delta\tilde x, 0)U^{\dsh}(\tilde x, 0)\, d\tilde x+\mathcal O(\delta^5)\\
&=\left\{\begin{aligned} &-\frac{c_n(n-2)}{2}\left(\mathbf H(p)\int_{\mathbb R^{n-1}}U^\dsh(\tilde x, 0)\, d\tilde x +\frac{\delta^2}{2}\int_{\mathbb R^{n-1}}\langle D^2 \mathbf H(p)\tilde x, \tilde x\rangle U^\dsh(\tilde x, 0)\, d\tilde x\right)+\mathcal O(\delta^3)\, \hbox{\tiny{if}}\, \mathbf H \, \hbox{\tiny{is not constant}}\\
&-\frac{c_n(n-2)}{2}\mathbf H\int_{\mathbb R^{n-1}}U^\dsh(\tilde x, 0)\, d\tilde x+\mathcal O(\delta^5)\, \hbox{\tiny{if}}\, \mathbf H \, \hbox{\tiny{is constant}}\end{aligned}\right.\end{aligned}$$
Analougously, using the fact that $p$ is a non-degenerate critical point of $\mathbf K$ when $\mathbf K$ is not constant

$$\begin{aligned}A_8&=-\frac{1}{\dst}\int_{\mathbb R^{n}_+}\mathbf K(\delta x)U^{\dst}(x)\, dx+\mathcal O(\delta^5)\\
&=\left\{\begin{aligned} &-\frac{1}{\dst}\left(\mathbf K(p)\int_{\mathbb R^{n}_+}U^\dst(x)\, dx +\frac{\delta^2}{2}\int_{\mathbb R^{n}_+}\langle D^2 \mathbf K(p)x, x\rangle U^\dst(x)\, dx\right)+\mathcal O(\delta^3)\, \hbox{\tiny{if}}\,\mathbf K \, \hbox{\tiny{is not constant}}\\
&-\frac{1}{\dst}\mathbf K\int_{\mathbb R^{n}_+}U^\dst(x)\, dx+\mathcal O(\delta^5)\, \hbox{\tiny{if}}\,\mathbf K \, \hbox{\tiny{is constant}}\end{aligned}\right.\end{aligned}$$
For the term $A_4$, expanding twice by Taylor formula we get
$$\begin{aligned}A_4&=\left\{\begin{aligned}&-\frac{c_n(n-2)}{2}\mathbf H(p)\int_{\mathbb R^{n-1}}\left[\mathfrak F(U+\delta^2 V_p)-\mathfrak F(U)\right]\, d\tilde x +\mathcal O(\delta^4)\, \hbox{\tiny{if}}\, \mathbf H\, \hbox{\tiny{is not constant}}\\
&-\frac{c_n(n-2)}{2}\mathbf H\int_{\mathbb R^{n-1}}\left[\mathfrak F(U+\delta^2 V_p)-\mathfrak F(U)\right]\, d\tilde x +\mathcal O(\delta^5)\, \hbox{\tiny{if}}\, \mathbf H\, \hbox{\tiny{is constant}}\end{aligned}\right.\end{aligned}$$
Then, if $H$ is constant then
$$\begin{aligned}A_4&=-\frac{c_n(n-2)}{2}\mathbf H\delta^2 \int_{\mathbb R^{n-1}}U^{\dsh-1}(\tilde x, 0)V_p(\tilde x, 0)\, d\tilde x\\
&-\frac{c_n(n-2)(\dsh-1)}{4}\mathbf H\delta^4\int_{\mathbb R^{n-1}}U^{{\dsh-2}}(\tilde x, 0) V_p^2(\tilde x, 0)\, d\tilde x+\mathcal O(\delta^5)\end{aligned}$$
while if $\mathbf H$ is not constant then
$$A_4=-\frac{c_n(n-2)}{2}\mathbf H(p)\delta^2 \int_{\mathbb R^{n-1}}U^{\dsh-1}(\tilde x, 0)V_p(\tilde x, 0)\, d\tilde x+\mathcal O(\delta^4).$$
Analogously, if $\mathbf K$ is constant then
$$\begin{aligned}A_7&=-\mathbf K\delta^2 \int_{\mathbb R^{n}_+}U^{\dst-1}(x)V_p(x)\, dx-\frac{\dst-1}{2}\mathbf K\delta^4\int_{\mathbb R^{n}_+}U^{{\dst-2}}(x) V_p^2(x)\, dx+\mathcal O(\delta^5)\end{aligned}$$
while if $\mathbf K$ is not constant then
$$A_7=-\mathbf K(p)\delta^2 \int_{\mathbb R^{n}_+}U^{\dst-1}(x)V_p(x)\, dx+\mathcal O(\delta^4).$$
At the end we evaluate $A_1$.
First we have that
$$\begin{aligned}A_1&=\frac{c_n}{2}\int_M |\nabla_{\tilde g_p}W_{\delta, p}|^2\,d\nu_{\tilde g_p}+c_n\delta^2\int_M \nabla_{\tilde g_p}W_{\delta, p}\nabla_{\tilde g_p}V_{\delta, p}\,d\nu_{\tilde g_p}+\frac{c_n}{2}\delta^4\int_M |\nabla_{\tilde g_p}V_{\delta, p}|^2\,d\nu_{\tilde g_p}\\
&=L_1+L_2+L_3.\end{aligned}$$
By using \eqref{detg} and \eqref{gij} and integrating by parts we get
$$\begin{aligned} L_3&=\frac{c_n}{2}\delta^4\int_{\mathbb R^n_+}|\nabla V_p|^2\, dx +\mathcal O(\delta^5)\\
&=-\frac{c_n}{2}\delta^4\int_{\mathbb R^n_+}V_p\Delta V_p\, dx +\frac{c_n}{2}\delta^4\int_{\partial\mathbb R^n_+}V_p\frac{\partial}{\partial \nu}V_p\, d\tilde x+\mathcal O(\delta^5)\\
&=-\frac{c_n}{2}\delta^4\int_{\mathbb R^n_+}V_p\Delta V_p\, dx +\frac{c_n n}{4}\delta^4\int_{\mathbb R^{n-1}}\mathbf H(p)U^{\dsh-2}V_p^2\, d\tilde x+\mathcal O(\delta^5)\end{aligned}$$ while
$$\begin{aligned}L_2&=c_n\delta^2\int_{\mathbb R^n_+}\nabla U \nabla V_p\, dx \\
&+\delta^4\int_{\mathbb R^n_+}\left(\frac 1 3 \bar R_{ikj\ell}\tilde x_k\tilde x_\ell \partial_i U\partial_j V_p+R_{ninj}x_n^2\partial_i U\partial_j V_p\right)\, dx +\mathcal O(\delta^5)\\
&=-c_n\delta^2\int_{\mathbb R^n_+}\Delta U V_p\, dx +c_n\delta^2\int_{\mathbb R^{n-1}}\frac{\partial}{\partial\nu}UV_p\, d\tilde x\\
&+\delta^4\int_{\mathbb R^n_+}\left(\frac 1 3 \bar R_{ikj\ell}\tilde x_k\tilde x_\ell \partial_i U\partial_j V_p+R_{ninj}x_n^2\partial_i U\partial_j V_p\right)\, dx +\mathcal O(\delta^5)\\
&=\delta^2\int_{\mathbb R^n_+}\mathbf K(p) U^{\dst-1}V_p\, dx +\frac{c_n(n-2)}{2}\delta^2\int_{\mathbb R^{n-1}}\mathbf H(p) U^{\dsh-1}V_p\, d\tilde x\\
&+\delta^4\underbrace{\int_{\mathbb R^n_+}\left(\frac 1 3 \bar R_{ikj\ell}\tilde x_k\tilde x_\ell \partial_i U\partial_j V_p+R_{ninj}x_n^2\partial_i U\partial_j V_p\right)\, dx}_{:=L_2^1} +\mathcal O(\delta^5)
\end{aligned}$$
Integrating by parts we get
$$\begin{aligned}L_2^1&=\underbrace{\int_{\partial\mathbb R^n_+}\left(\frac 1 3 \bar R_{ikj\ell}\tilde x_k\tilde x_\ell+R_{ninj}x_n^2\right)V_p\partial_i U\nu_j}_{:=0\, \hbox{\tiny{since}}\, \nu_j=0\, j=1, \ldots n-1}-\int_{\mathbb R^n_+}\underbrace{\left(\frac 1 3 \bar R_{ikj\ell}\tilde x_k\tilde x_\ell+R_{ninj}x_n^2\right)\partial^2_{ij}U}_{:=\mathtt E_p}V_p\\
&-\int_{\mathbb R^n_+}\partial_j \left(\frac 1 3 \bar R_{ikj\ell}\tilde x_k\tilde x_\ell+R_{ninj}x_n^2\right)\partial_i U V_p\\
&=-\int_{\mathbb R^n_+}\mathtt E_p V_p-\frac 13\bar R_{i\ell}\int_{\mathbb R^n_+}\tilde x_\ell \partial_i U V_p-\frac 13 \bar R_{ikjj}\int_{\mathbb R^n_+}\tilde x_l\partial_i U V_p\\
&=-\int_{\mathbb R^n_+}\mathtt E_p V_p\end{aligned}$$
by using the symmetries of the curvature tensor and \eqref{Rij}. Hence
$$\begin{aligned}L_2&=\delta^2\int_{\mathbb R^n_+}\mathbf K(p) U^{\dst-1}V_p\, dx +\frac{c_n(n-2)}{2}\delta^2\int_{\mathbb R^{n-1}}\mathbf H(p) U^{\dsh-1}V_p\, d\tilde x-\delta^4\int_{\mathbb R^n_+}\mathtt E_p V_p+\mathcal O(\delta^5)\\
\end{aligned}$$ 
Finally, by \eqref{detg}, \eqref{gij} and since the terms of odd degree disappear by symmetry we get
$$\begin{aligned}L_1&=\frac{c_n}{2}\int_{\mathbb R^n_+}|\nabla U|^2+\frac{c_n}{2}\delta^2\int_{\mathbb R^n_+}\left(\frac 1 3 \bar R_{ikj\ell}\tilde x_k\tilde x_\ell+R_{ninj}x_n^2\right)\partial_i U\partial_j U\\
&+\frac{c_n}{2}\delta^4\int_{\mathbb R^n_+}\left(\frac{1}{20}\bar R_{ikj\ell, mp}+\frac{1}{15}\bar R_{iks\ell}\bar R_{jmsp}\right)\tilde x_k\tilde x_\ell \tilde x_m \tilde x_p \partial_i U \partial_j U\\
&+\frac{c_n}{2}\delta^4\underbrace{\int_{\mathbb R^n_+}\left(\frac 12 R_{ninj, k\ell}+\frac 13 {\rm Sym}_{ij}(\bar R_{iks\ell}R_{nsnj})\right)x_n^2\tilde x_k \tilde x_\ell \partial_i U \partial_j U}_{G_1}\\
&+\frac{c_n}{2}\delta^4\underbrace{\int_{\mathbb R^n_+}\left(\frac 13 R_{ninj, nk}x_n^3\tilde x_k+\frac{1}{12}(R_{ninj, nn}+8R_{nins}R_{nsnj})x_n^4\right)\partial_i U \partial_j U}_{:=G_2}\\
&+\mathcal O(\delta^5).\end{aligned}$$
 Reasoning as in the proof of Lemma 6 in \cite{GMP} one can show that all the terms of order $\delta^2$ vanish.\\ Moreover, by the symmetries of the curvature tensor (see \cite{Marques}, page 1614 formula C) we get
$$\int_{\mathbb R^n_+}\left(\frac{1}{20}\bar R_{ikj\ell, mp}+\frac{1}{15}\bar R_{iks\ell}\bar R_{jmsp}\right)\tilde x_k\tilde x_\ell\tilde x_m\tilde x_p\partial_i U\partial_j U=0.$$
Moreover
$$\begin{aligned} G_2&=\frac{\alpha_n^2(n-2)^2}{12(n-1)|\mathbf K(p)|^{\frac{n-2}{2}}}\int_{\mathbb R^n_+}\left(\underbrace{R_{nini,nn}}_{:=R_{nn,nn=-2R_{nins}^2}}+8\underbrace{R_{nins}R_{nsni}}_{:=R_{nins}^2}\right)\frac{x_n^4|\tilde x|^2}{(|\tilde x|^2+(x_n+\D(p)^2-1)^n}\, dx\\
&=\frac{\alpha_n^2(n-2)^2}{2(n-1)|\mathbf K(p)|^{\frac{n-2}{2}}}R_{nins}^2\int_{\mathbb R^n_+}\frac{x_n^4|\tilde x|^2}{(|\tilde x|^2+(x_n+\D(p)^2-1)^n}\, dx
\end{aligned}$$ 
It remains to estimate $G_1$. By symmetry reasons we have only to consider the cases $i=j=k=\ell$, $i=j\neq k=\ell$, $i=k\neq j=\ell$ and $i=\ell\neq j=k$. Then the Symbol term gives no contribution. 
Hence $$\begin{aligned}G_1&=\frac{\alpha_n^2(n-2)^2}{|\mathbf K(p)|^{\frac{n-2}{2}}}\int_{\mathbb R^n_+}R_{ninj, k\ell}\frac{x_n^2\tilde x_k\tilde x_\ell\tilde x_i\tilde x_j}{(|\tilde x|^2+(x_n+\D)^2-1)^n}\, d\tilde x d x_n\\
&=\frac{\alpha_n^2(n-2)^2}{|\mathbf K(p)|^{\frac{n-2}{2}}}\sum_i R_{nini,ii}\int_{\mathbb R^n_+}\frac{x_n^2\tilde x_i^4}{(|\tilde x|^2+(x_n+\D)^2-1)^n}\, d\tilde x d x_n\\
&+\left(\sum_{i\neq k}R_{nini,kk}+\sum_{i\neq j}R_{ninj,ij}+\sum_{i\neq j}R_{ninj,ji}\right)\int_{\mathbb R^n_+}\frac{x_n^2\tilde x_i^2\tilde x_j^2}{(|\tilde x|^2+(x_n+\D)^2-1)^n}\\
&=\frac{\alpha_n^2(n-2)^2}{|\mathbf K(p)|^{\frac{n-2}{2}}}\left[\sum_i R_{nini,ii}+\frac 13\left(\sum_{i\neq k}R_{nini,kk}+\sum_{i\neq j}R_{ninj,ij}+\sum_{i\neq j}R_{ninj,ji}\right)\right]\int_{\mathbb R^n_+}\frac{x_n^2\tilde x_i^4}{(|\tilde x|^2+(x_n+\D)^2-1)^n}\\
&=\frac{\alpha_n^2(n-2)^2}{3|\mathbf K(p)|^{\frac{n-2}{2}}}\left[3\sum_i R_{nini,ii}+\sum_{i\neq k}R_{nini,kk}+\sum_{i\neq j}R_{ninj,ij}+\sum_{i\neq j}R_{ninj,ji}\right]\int_{\mathbb R^n_+}\frac{x_n^2\tilde x_i^4}{(|\tilde x|^2+(x_n+\D)^2-1)^n}\\
&=\frac{\alpha_n^2(n-2)^2}{(n^2-1)|\mathbf K(p)|^{\frac{n-2}{2}}}\left[3\sum_i R_{nini,ii}+\sum_{i\neq k}R_{nini,kk}+\sum_{i\neq j}R_{ninj,ij}+\sum_{i\neq j}R_{ninj,ji}\right]\int_{\mathbb R^n_+}\frac{x_n^2|\tilde x|^4}{(|\tilde x|^2+(x_n+\D)^2-1)^n}\\
\end{aligned}$$
Here we have used the fact that (see \cite{Marques})
$$\int_{\mathbb R^n_+}\frac{x_n^2\tilde x_i^2\tilde x_j^2}{(|\tilde x|^2+(x_n+\D)^2-1)^n}\, d\tilde x d x_n=\frac 13 \int_{\mathbb R^n_+}\frac{x_n^2\tilde x_i^4}{(|\tilde x|^2+(x_n+\D)^2-1)^n}\, d\tilde x d x_n$$
and
$$\int_{\mathbb R^n_+}\frac{x_n^2\tilde x_i^4}{(|\tilde x|^2+(x_n+\D)^2-1)^n}\, d\tilde x d x_n=\frac{3}{n^2-1}\int_{\mathbb R^n_+}\frac{x_n^2|\tilde x|^4}{(|\tilde x|^2+(x_n+\D)^2-1)^n}\, d\tilde x d x_n.$$
At this point we have also that $R_{nn, kk}=0$ for all $k=1, \ldots, n-1$ (see Proposition 3.2 of \cite{Marques}). Then, at then end, we get
$$G_1=\frac{\alpha_n^2(n-2)^2}{(n^2-1)|\mathbf K(p)|^{\frac{n-2}{2}}}R_{ninj,ij}\int_{\mathbb R^n_+}\frac{x_n^2|\tilde x|^4}{(|\tilde x|^2+(x_n+\D)^2-1)^n}$$
Collecting all the estimates we get that
$$\begin{aligned}L_1&=\frac{c_n}{2}\int_{\mathbb R^n_+}|\nabla U|^2+\\
&+ \frac{c_n\alpha_n^2(n-2)^2}{2(n-1)|\mathbf K(p)|^{\frac{n-2}{2}}}\delta^4\left[\frac 12 R^2_{nins}\int_{\mathbb R^n_+}\frac{x_n^4|\tilde x|^2}{(|\tilde x|^2+(x_n+\D(p)^2-1)^n}\, dx+\frac{1}{n+1}\int_{\mathbb R^n_+}\frac{x_n^2|\tilde x|^4}{(|\tilde x|^2+(x_n+\D)^2-1)^n}\right]\\&+\mathcal O(\delta^5).\end{aligned}$$
So, if $\mathbf K$ and $\mathbf H$ are not constants, then we remark again that the correction $V_p$ is not necessary and the reduced functional is (putting together the previous estimates and letting $\gamma=1$) 
$$J_{\e, g}(\tilde\Theta_{\delta, p})=\mathfrak E(p)+\mathtt A(p)\e \delta -\mathtt B(p)\delta^2+\mathcal O(\delta^3)+\mathcal O(\e\delta^3)$$
where $\mathfrak E(p)$ is the energy of the bubble evaluated in Lemma \ref{energiabubble} while
\begin{equation}\label{ap}\mathtt A(p):=(n-1)\int_{\mathbb R^{n-1}}U^2(\tilde x, 0)\, d\tilde x\end{equation}
and \begin{equation}\label{bp}\mathtt B(p):=\frac{c_n(n-2)}{4}\int_{\mathbb R^{n-1}}\langle D^2\mathbf H(p)\tilde x, \tilde x\rangle U^\dsh(\tilde x, 0)\, d\tilde x +\frac{1}{2\cdot\dst}\int_{\mathbb R^n_+}\langle D^2\mathbf K(p)\tilde x, \tilde x\rangle U^\dst.\end{equation}
If, instead, $\mathbf  H$ and $\mathbf  K$ are constants and using the identity $$\partial^2_{nn}k_{\tilde g_p}=-2R_{ninj,ij}-2R^2_{ninj}$$ then we have that
$$\begin{aligned}
&J_{\e, g}(\tilde\Theta_{\delta, g})=\mathfrak E+\frac 12 \delta^4\int_{\mathbb R^n_+}\left(c_n\Delta V_p +(\dst-1)\mathbf K(p)U^{\dst-2}V_p\right)V_p\, dx+(n-1)\gamma(p)\e\delta\int_{\mathbb R^{n-1}}U^2(\tilde x, 0)\, d\tilde x\\
&-\delta^4\frac{1}{24(n-1)}|\overline{{\rm Weyl}_g}(p)|^2 \int_{\mathbb R^n_+}|\tilde x|^2 U^2(\tilde x, x_n)\, d\tilde x+\frac 14 \delta^4 \partial^2_{nn}k_{\tilde g}\int_{\mathbb R^n_+}x_n^2U^2\, dx + \delta^4\frac{c_n\alpha_n^2(n-2)^2}{2(n-1)|\mathbf K(p)|^{\frac{n-2}{2}}}\times\\
&\times\left(\frac 12 R^2_{nins}\int_{\mathbb R^n_+}\frac{x_n^4|\tilde x|^2}{(|\tilde x|^2+(x_n+\D(p)^2-1)^n}\, dx+\frac{1}{n+1}R_{ninj,ij}\int_{\mathbb R^n_+}\frac{x_n^2|\tilde x|^4}{(|\tilde x|^2+(x_n+\D)^2-1)^n}\right)\\
&+\mathcal O(\delta^5)\\
&=\mathfrak E+\frac 12 \delta^4\int_{\mathbb R^n_+}\left(c_n\Delta V_p +(\dst-1)\mathbf K(p)U^{\dst-2}V_p\right)V_p\, dx+(n-1)\gamma(p)\e\delta\int_{\mathbb R^{n-1}}U^2(\tilde x, 0)\, d\tilde x\\
&-\delta^4\frac{1}{24(n-1)}|\overline{{\rm Weyl}_g}(p)|^2 \int_{\mathbb R^n_+}|\tilde x|^2 U^2(\tilde x, x_n)\, d\tilde x\\
&+\delta^4 R^2_{nins}\underbrace{\left(\frac{c_n\alpha_n^2(n-2)^2}{4(n-1)|\mathbf K(p)|^{\frac{n-2}{2}}}\int_{\mathbb R^n_+}\frac{x_n^4|\tilde x|^2}{(|\tilde x|^2+(x_n+\D(p)^2-1)^n}\, dx-\frac 12 \int_{\mathbb R^n_+}x_n^2U^2\, dx\right)}_{(I_1)}\\
&+\delta^4R_{ninj,ij}\underbrace{\left( \frac{c_n\alpha_n^2(n-2)^2}{2(n^2-1)|\mathbf K(p)|^{\frac{n-2}{2}}}\int_{\mathbb R^n_+}\frac{x_n^2|\tilde x|^4}{(|\tilde x|^2+(x_n+\D)^2-1)^n}-\frac 12 \int_{\mathbb R^n_+}x_n^2U^2\, dx\right)}_{(I_2)}
\end{aligned}$$
First we remark that by simply evaluate
$$\begin{aligned}
\int_{\mathbb R^n_+}x_n^2U^2\, dx&=\frac{\alpha_n^2}{|\mathbf K|^{\frac{n-2}{2}}}\int_0^{+\infty}\int_{\mathbb R^{n-1}}\frac{x_n^2}{\left(|\tilde x|^2+(x_n+\D)^2-1\right)^{n-2}}\, d\tilde x\, dx_n\\
&=\omega_{n-1}\frac{\alpha_n^2}{|\mathbf K|^{\frac{n-2}{2}}}\int_{\D}^{+\infty}\frac{(t-\D)^2}{(t^2-1)^{\frac{n-3}{2}}}\, dtI_{n-2}^{n-2}\\
&=\omega_{n-1}\frac{\alpha_n^2}{|\mathbf K|^{\frac{n-2}{2}}}\hat\varphi_{\frac{n-3}{2}}\frac{4(n-2)}{n+1}I_n^{n+2}.\end{aligned}$$
Instead 
$$\begin{aligned}
\int_{\mathbb R^n_+}\frac{x_n^2|\tilde x|^4}{(|\tilde x|^2+(x_n+\D(p)^2-1)^n}\, dx&=\omega_{n-1}I^{n+2}_n \int_{\D}^{+\infty}\frac{(t-\D)^2}{(t^2-1)^{\frac{n-3}{2}}}\, dt =\omega_{n-1}I^{n+2}_{n}\hat\varphi_{\frac{n-3}{2}}.\end{aligned}$$ At the end
$$\begin{aligned}
\int_{\mathbb R^n_+}\frac{x_n^4|\tilde x|^2}{(|\tilde x|^2+(x_n+\D(p)^2-1)^n}\, dx=\omega_{n-1}\frac{n-3}{n+1} I_n^{n+2}\tilde\varphi_{\frac{n-1}{2}}.\end{aligned}$$ An integration by parts shows that
$$\tilde\varphi_{\frac{n-1}{2}}=\frac{3}{n-3}\hat\varphi_{\frac{n-3}{2}}-\D\int_{\D}^{+\infty}\frac{(t-\D)^3}{(t^2-1)^{\frac{n-1}{2}}}.$$
Then
\begin{equation}\label{s}\begin{aligned}(I_1)&:=\frac{\alpha_n^2}{|\mathbf K|^{\frac{n-2}{2}}}\omega_{n-1}\frac{n-2}{n+1}I^{n+2}_n \left((n-3)\tilde\varphi_{\frac{n-1}{2}}-4\hat\varphi_{\frac{n-3}{2}}\right)\\
&=\frac{\alpha_n^2}{|\mathbf K|^{\frac{n-2}{2}}}\omega_{n-1}\frac{n-2}{n+1}I^{n+2}_n \left(-\hat\varphi_{\frac{n-3}{2}}-(n-3)\D\int_{\D}^{+\infty}\frac{(t-\D)^3}{(t^2-1)^{\frac{n-1}{2}}}\right)\\
&=-\mathtt S<0\end{aligned}\end{equation}
while
$$\begin{aligned}(I_2)&:=0\end{aligned}$$
Then

$$J_{\e, g}(\tilde\Theta_{\delta, g})=\mathfrak E+\mathtt A\gamma(p)\e\delta  -\delta^4\mathtt B(p)+\mathcal O(\delta^5)$$ where
$\mathfrak E$ is the energy of the bubble that does not depend on the point $p$, $\mathtt A\equiv \mathtt A(p)$ is defined as in \eqref{ap}, while now
\begin{equation}\label{bp1}\begin{aligned}\mathtt B(p)&:=-\frac 12 \int_{\mathbb R^n_+}\left(c_n\Delta V_p +(\dst-1)\mathbf K(p)U^{\dst-2}V_p\right)V_p\, dx \\&+\left(\frac{1}{24(n-1)}|\overline{{\rm Weyl}_g}(p)|^2 \int_{\mathbb R^n_+}|\tilde x|^2 U^2(\tilde x, x_n)\, d\tilde x+ R^2_{nins}\mathtt S\right)\end{aligned}\end{equation}

 \end{proof}
Now we are ready to prove Theorem \ref{principale}. 
\begin{proof}[Proof of Theorem \ref{principale}]
If $\mathbf K$ and $\mathbf H$ are constants we let $\delta=d\e^{\frac 13}$, $d\in [\alpha,\beta]\subset (0, +\infty)$.\\ By summarizing the previous results we have that $$J_{\e, g}\left(\tilde\Theta_{d\e^{\frac 13}, p}+\tilde\Phi\right)=\mathfrak E +\e^{\frac 4 3}\left(\mathtt A \gamma(p) d-d^4 \mathtt B(p)\right)+\mathcal O(\e^{\frac 53})$$ $C^0-$ uniformly for $p\in\partial M$, $d\in [\alpha,\beta]$, where $\mathtt A, \mathtt B(p), \mathfrak E$ are defined in Lemma \ref{redf}.\\ Now we let the reduced functional $$\mathcal F_\e(d, p)=J_{\e, g}\left(\tilde\Theta_{d\e^{\frac 13}, p}+\tilde\Phi\right).$$ It is standard to show that if $(\bar d, \bar p)\in (0, +\infty)\times \partial M$ is a critical point of the $\mathcal F_\e(d, p)$ then $\tilde\Theta_{d\e^{\frac 13}, p}+\tilde\Phi$ is a solution of \eqref{pb1}.\\ We let now $$\mathcal G(d, p)=\mathtt A \gamma(p) d-d^4 \mathtt B(p)$$ where $\mathtt A>0$ while $\mathtt B(p)>0$ by the hypothesis of Theorem \ref{principale}.\\ Then, one can check that there exist $0<\alpha<\beta$ such that any critical point $(d, p)\in (0, +\infty)\times \partial M$ of $\mathcal G$ lies in $(\alpha, \beta)\times \partial M$ because $$\frac{\partial\mathcal G}{\partial d}=\mathtt A \gamma(p)-4d^3\mathtt B(p)$$ and $$ \frac{\partial\mathcal G}{\partial d}(d, p)=0 \quad \mbox{if and only if}\quad d^3=\frac{\gamma(p)}{\mathtt B(p)}>0.$$ Moreover for any $L<0$ there exists $\bar d>0$ such that $\mathcal G(d, p)<L$ for any $d>\bar d$ and for any $p\in\partial M$.\\ Then there exists a maximum point $(d_0, p_0)\in (\alpha,\beta)\times \partial M$ which is $C^0-$ stable.\\\\
If, instead $\mathbf H$ and $\mathbf K$ are not constants then we let $\delta=d\e$, $d\in [\alpha,\beta]\subset (0, +\infty)$.\\ By summarizing the previous results we have that $$J_{\e, g}\left(\tilde\Theta_{d\e, p}+\tilde\Phi\right)=\mathfrak E(p) +\e^{2}\left(\mathtt A(p) d-d^2 \mathtt B(p)\right)+\mathcal O(\e^{3})$$ $C^0-$ uniformly for $p\in\partial M$, $d\in [\alpha,\beta]$, where $\mathtt A(p), \mathtt B(p), \mathfrak E(p)$ are defined in Lemma \ref{redf}.\\ We again define the reduced functional $$\mathcal F_\e(d, p)=J_{\e, g}\left(\tilde\Theta_{d\e, p}+\tilde\Phi\right).$$ Now we set $\mathtt G_p(d)=d \mathtt A(p)-d^2 \mathtt B(p)$. Let
$p_0\in\partial M$ be a non-degenerate minimum point of
$H$ and
$K$ in the sense of the assumption ${\rm (Hyp)_2}$. \\ By Lemma \ref{energiabubble}, it is easy to see that
$p_0$ is a non-degenerate maximum point of
$\mathfrak E(p)$.

 Hence, there is a
$\sigma_1-$ neighbourhood of
$p_0$, say
$\mathcal U_{\sigma_1}\subset \partial M$, such that for any sufficiently small
$\gamma>0$ \begin{equation}\label{Ep}\mathfrak E(p)\leq \mathfrak E(p_0)-\gamma\quad \forall\,\, p\in \partial U_{\sigma_1}.\end{equation} Now we see that
\begin{equation}\label{choiced} d_0:=\frac{\mathtt A(p_0)}{2\mathtt B(p_0)}\end{equation}
is a strictly maximum point of the function
$\mathtt G_{p_0}(d)$. Then there is an open interval
$I_{\sigma_2}$ such that
$\bar I_{\sigma_2}\subset \mathbb R^+$ and 
\begin{equation}\label{Gp}
\mathtt G_{p_0}(d)\leq \mathtt G_{p_0}(d_0)-\gamma\quad \forall\,\, d\in\partial I_{\sigma_2}.\end{equation}
Let us set
$\mathcal K:=\overline{\mathcal U_{\sigma_1}\times I_{\sigma_2}}$ and let
$\eta>0$ be small enough so that
$\mathcal K\subset \mathcal U_{\sigma_1}\times 
  ( \eta, \frac 1 \eta
  ) $. Since the reduced functional is continuous on
$\mathcal K$ then, by Weierstrass Theorem it follows that it has a global maximum point in
$\mathcal K$. Let
$(p_\e, d_\e)$ such point. We want to show that it is in the interior of
$\mathcal K$.

By contradiction suppose that the point
$(p_{\e}, d_{\e})\in\partial \mathcal K$. There are two possibilities: \begin{itemize}\item[(a)]
$p_{\e}\in\partial\mathcal U_{\sigma_1}$,
$d_{\e}\in \bar I_{\sigma_2}$ \item[(b)] 
$p_{\e}\in\mathcal U_{\sigma_1}$,
$d_{\e}\in \partial I_{\sigma_2}$.\end{itemize}

If (a) holds, by using the fact that $(p_\e, d_\e)$ is a maximum point for $\mathcal F_\e$, Lemma \ref{redf} and \eqref{Ep}
 we have 
\begin{equation*} 
0\leq \mathcal F_\e(p_\e, d_\e)-
\mathcal F_\e(p_0, d_\e)=
\mathfrak E(p_\e)-\mathfrak E(p_0)
+\mathcal O(\e^2)\leq -\gamma 
+\mathcal O(\e^2)<0
\end{equation*} 
for 
$\e$ sufficiently small, 
which is a contradiction.

If now (b) holds, then by using 
Lemma \ref{redf}, again the fact that
$(p_\e, d_\e)$ is a maximum point for
$\mathcal F_\e$ and \eqref{Gp},
 we have 
\begin{align}
\label{jp1}
0 \leq \mathcal F_\e(p_\e, d_\e)-\mathcal F_\e(p_\e, d_0)
&
=\e^2
\big ( \mathtt G_{p_\e}(d_\e)-\mathtt G_{p_\e}(d_0)+o(1)
\big ) 
\leq -\gamma\e^2+o(\e^2)<0 
\end{align} 
for any
$\e$ sufficiently small which is again a contradiction.

 It remains to show that
$(p_\e, d_\e)\to (p_0, d_0)$ as
$\e\to 0$. Indeed,
by using the fact that
$(p_\e, d_\e)$ is a maximum point for
$\mathcal F_\e$ and Lemma \ref{redf} we get
$$\mathcal F_\e(p_0, d_\e)\leq \mathcal F_\e(p_\e, d_\e) \quad\iff\quad \mathfrak E(p_0)\leq \mathfrak E(p_\e).$$ Moreover by \eqref{Ep}
$$\mathfrak E(p_\e)\leq \mathfrak E(p_0)$$ and hence, passing to the limit it follows
$$\lim_{\e\to 0}\mathfrak E(p_\e)=\mathfrak E(p_0).$$ Up to a subsequence, since
$p_\e$ is a local maximum for
$\mathfrak E$ it follows that
$p_\e\to p_0$.

In the same way one can show that
$d_\e\to d_0$ as
$\e\to 0$.

\end{proof}

\end{document}